\documentclass[10.5pt]{article}
\usepackage[affil-it]{authblk}

\usepackage[top=30truemm,bottom=30truemm,left=30truemm,right=30truemm]{geometry}

\usepackage{amssymb, amsmath}
\usepackage{amsthm}

\makeatletter
    
    \@addtoreset{equation}{section}
  \makeatother

\newtheorem{definition}{Definition}[section]
\newtheorem{proposition}[definition]{Proposition}
\newtheorem{theorem}[definition]{Theorem}
\newtheorem{lemma}[definition]{Lemma}

\newtheorem{remark}{Remark}[section]

\begin{document}

\title{%
{\bf Scaling variables and asymptotic profiles for
the semilinear damped wave equation with variable coefficients}
}%

\author{Yuta Wakasugi}
\date{\empty}
 \affil{Graduate School of Mathematics, Nagoya University,\\
 Furocho, Chikusaku, Nagoya 464-8602, Japan\\
 Email: {\rm yuta.wakasugi@math.nagoya-u.ac.jp}}

\maketitle

\begin{abstract}
We study the asymptotic behavior of solutions for the semilinear damped wave equation
with variable coefficients.
We prove that if the damping is effective, and the nonlinearity and other lower order terms
can be regarded as perturbations, then the solution is approximated by
the scaled Gaussian of the corresponding linear parabolic problem.
The proof is based on the scaling variables and energy estimates.
\end{abstract}

\section{Introduction}

We consider the Cauchy problem of the semilinear damped wave equation
with lower order perturbations
\begin{align}
\label{dw}
	\left\{ \begin{array}{ll}
	u_{tt}+b(t)u_{t}
	=\Delta_{x}u + c(t)\cdot \nabla_{x}u + d(t)u + N(u,\nabla_{x}u,u_{t}),&
	t>0,x\in\mathbb{R}^n,\\
	u(0,x)=\varepsilon u_0(x),\ u_{t}(0,x)=\varepsilon u_1(x),
	&x \in\mathbb{R}^n,
	\end{array}\right.
\end{align}
where the coefficients $b, c$ and $d$ are smooth,
$b$ satisfies
\begin{align}
\label{ef_dm}
	b(t) \sim (1+t)^{-\beta},\quad -1\le \beta<1,
\end{align}
and
$c(t)\cdot \nabla_{x}u, d(t)u, N(u,\nabla_{x}u,u_{t})$
can be regarded as perturbations
(the precise assumption will be given in the next section).
Also, $\varepsilon$ denotes a small parameter.

Our purpose is to give the asymptotic profile of global solutions to \eqref{dw} with small initial data
as time tends to infinity.
By the assumption \eqref{ef_dm}, the damping is effective, and
we can expect that the asymptotic profile of solutions is given by
the scaled Gaussian (see \eqref{def_B}, \eqref{gauss} and \eqref{sol_asym}).

The existence of global solutions and the asymptotic behavior of solutions
to damped wave equations
have been widely investigated for a long time.
Matsumura \cite{Ma76} obtained decay estimates of solutions to
the linear damped wave equation
\begin{align}
\label{linear_DW}
	u_{tt}-\Delta u + u_t = 0,
\end{align}
and applied them to nonlinear problems.
After that,
Yang and Milani \cite{YaMi00} showed that
the solution of \eqref{linear_DW} has the so-called
{\em diffusion phenomena}, that is,
the asymptotic profile of solutions to \eqref{linear_DW} is given by
the Gaussian in the $L^{\infty}$-sense.
Marcati and Nishihara \cite{MaNi03} and Nishihara \cite{Ni03MathZ}
gave more detailed informations about the asymptotic behavior of solutions.
They found that when $n=1,3$, the solution of \eqref{linear_DW}
is asymptotically decomposed
into the Gaussian and a solution of the wave equation
(with an exponentially decaying coefficient) in the $L^p$--$L^q$ sense
(see Hosono
and 
Ogawa \cite{HoOg04} and Narazaki \cite{Na04} for $n=2$ and $n\ge 4$).

For the nonlinear problem
\begin{align}
\label{nonlinear_DW}
	\left\{\begin{array}{l}
	u_{tt}-\Delta u + u_t = N(u),\\
	(u,u_t)(0,x)=\varepsilon (u_0,u_1)(x),
	\end{array}\right.
\end{align}
there are many results about global existence and asymptotic behavior of solutions
(see for example, \cite{IkNiZh06, IkTa05, Ka00, KaUe13, KaNaOn95, LiZh95, Ni06}).
In particular, Todorova and Yordanov \cite{ToYo01}
and Zhang \cite{Zh01} 
proved that
when $N(u)=|u|^p$, the critical exponent of \eqref{nonlinear_DW} is given by
$p = 1+2/n$.
More precisely, they showed that,
for initial data satisfying
$(u_0, u_1) \in H^{1,0}(\mathbb{R}^n) \times L^2(\mathbb{R}^n)$ and having
compact support,
if $p> 1+2/n$, then the global solution uniquely exists for small $\varepsilon$;
if $p \le 1+2/n$ and $\int_{\mathbb{R}^n} (u_0+u_1)(x)dx > 0$,
then the local-in-time solution blows up in finite time for any $\varepsilon > 0$.
The number $1+2/n$ is the same as
the well-known Fujita exponent,
which is the critical exponent of the semilinear heat equation
$v_t - \Delta v = v^p$ (see \cite{Fu66}),
though the role of the critical exponent is
different in the semilinear heat equation and the semilinear damped wave equation.
In fact, for the subcritical case $1<p<1+2/n$,
the solution of the semilinear damped wave equation blows up in finite time
under the positive mass condition
$\int_{\mathbb{R}^n} (u_0 + u_1)(x) dx > 0$,
while all positive solutions blow up in finite time for the semilinear heat equation.

Concerning the asymptotic behavior of the global solution, 
Hayashi, Kaikina and Naumkin \cite{HaKaNa04DIE} proved that
if $N$ satisfies $|N(u)| \le C |u|^p$ with $p>1+2/n$,
then the unique global solution exists for suitably small data and
the asymptotic profile of the solution is given by a constant multiple of the Gaussian.
However, they used the explicit formula of the fundamental solution of the
linear problem in the Fourier space, and hence,
it seems to be difficult to apply their method to variable coefficient cases.

Gallay and Raugel \cite{GaRa98} considered
the one-dimensional damped wave equation with variable principal term and a constant damping
\begin{align*}
	u_{tt} - ( a(x) u_{x} )_x + u_t = N(u,u_x,u_t).
\end{align*}
They used scaling variables
\begin{align}
\label{gr_sv}
	s=\log (t+t_0),\quad y =\frac{x}{\sqrt{t+t_0}},
\end{align}
and showed that if $a(x)$ is positive and has the positive limits $\lim_{x\to\pm\infty}a(x)=a_{\pm}$,
then the solution can be asymptotically expanded in terms of the corresponding parabolic equation.
Moreover, this expansion can be determined up to the second order.
Recently, Takeda \cite{Ta15, Ta16}
and Kawakami and Takeda \cite{KaTa}
obtained the complete expansion
for the linear and nonlinear damped wave equation with constant coefficients.

The wave equation with variable coefficient damping
\[
	u_{tt} - \Delta u + b(t,x)u_t = 0
\]
has been also intensively studied.
Yamazaki \cite{Ya06, Ya07} and Wirth \cite{Wi04, Wi06, Wi07JDE, Wi07ADE}
considered time-dependent damping $b=b(t)$.
Here we briefly explain their results
by restricting the damping $b$ to $b(t)= \mu (1+t)^{-\beta}$
with
$\mu > 0$ and $\beta \in \mathbb{R}$,
although they discussed more general $b(t)$:
(i) when $\beta>1$ (scattering), the solution scatters to a solution of the free wave equation;
(ii) when $\beta=1$ (non-effective weak dissipation),
the behavior of solutions depends on the constant $\mu$, and
the solution scatters with some modification;
(iii) when $\beta \in [-1,1)$ (effective),
the asymptotic profile of the solution is given by the scaled Gaussian;
(iv) when $\beta <-1$ (overdamping),
the solution tends to some asymptotic state,
which 
is nontrivial function for nontrivial data.
Hence our assumption \eqref{ef_dm} is reasonable
because the asymptotic behavior of solutions to the linear problem completely changes
when $\beta <-1$ or $\beta \ge 1$.

In the space-dependent damping case $b=b(x)=(1+|x|^2)^{-\alpha/2}$,
Mochizuki \cite{Mo76} (see also \cite{MoNa96}) proved that
if $\alpha >1$, then the energy of solution does not decay to zero in general
and solutions with data satisfying certain condition scatter to free solutions.
On the other hand, Todorova and Yordanov \cite{ToYo09} obtained
energy decay of solutions when $\alpha \in [0,1)$ and
the decay rates agree with those of the corresponding parabolic equation.
Moreover, the author of this paper \cite{Wa14JHDE} proved that the solution
actually has the diffusion phenomena when $\alpha \in [0,1)$.
In the critical case $\alpha = 1$, that is, $b=\mu (1+|x|^2)^{-1/2}$,
Ikehata, Todorova and Yordanov \cite{IkToYo13}
obtained optimal decay estimates of the energy of solutions
and found that the decay rate depends on the constant $\mu$.
However, the precise asymptotic profile is still open.
On the other hand,
Radu, Todorova and Yordanov \cite{RaToYo11, RaToYo16}
studied the diffusion phenomena
for solutions to the abstract damped wave equation
\[
	( L \partial_t^2 + \partial_t + A ) u = 0 
\]
by the method of the diffusion approximation,
where $A$ is a nonnegative self-adjoint operator, and
$L$ is a bounded positive self-adjoint operator. 
Recently, Nishiyama \cite{Nis15} studied the abstract damped wave equation
having the form
$( \partial_t^2 + M \partial_t + A ) u = 0$,
where $M$ is a bounded nonnegative self-adjoint operator. 
Moreover, as an application, he also determined the asymptotic profile
of solutions to the damped wave equation with variable coefficients
under a geometric control condition.

For the semilinear wave equation with space-dependent damping
\[
	u_{tt}-\Delta u +b(x)u_t = N(u),
\]
Ikehata, Todorova and Yordanov \cite{IkToYo09}
proved that when
$b(x) \sim (1+|x|)^{-\alpha}$ with $\alpha \in [0,1)$
and $N(u)=|u|^p$,
the critical exponent is $p=1+2/(n-\alpha)$
(see also Nishihara \cite{Ni10} for the case
$N(u)=-|u|^{p-1}u$
and
$b(x)=(1+|x|^2)^{-\alpha/2}$ with $\alpha\in [0,1)$).

Recently, the asymptotic behavior of solutions to the semilinear wave equation
with time-dependent damping
\[
	u_{tt} - \Delta u + b(t) u_t = N(u)
\]
was also studied.
When $b(t)=(1+t)^{-\beta}\ (-1<\beta<1)$ and $N(u)=|u|^p$,
Lin, Nishihara and Zhai \cite{LiNiZh12} determined the critical exponent as
$p=1+2/n$,
provided that the initial data belong to
$H^{1,0}(\mathbb{R}^n) \times L^2(\mathbb{R}^n)$ with compact support.
D'Abbicco, Lucente and Reissig \cite{DaLuRe13} (see also \cite{DaLu13})
extended this result to more general
$b(t)$ satisfying a monotonicity condition and a polynomial-like behavior.
Moreover, they relaxed the assumption on the data
to exponentially decaying condition.
They also dealt with the initial data belong to the class
$(L^1(\mathbb{R}^n) \cap H^{1,0}(\mathbb{R}^n))
\times (L^1(\mathbb{R}^n) \cap L^2(\mathbb{R}^n))$ when $n\le 4$.
We also refer the reader to D'Abbicco \cite{Da15} for the critical case $\beta=1$.
On the other hand,
Nishihara \cite{Ni11AA} studied the asymptotic profile of solutions
in the case
$n=1, b=(1+t)^{-\beta}\ (-1<\beta<1)$,
$(u_0,u_1)\in H^{1,0}(\mathbb{R}^n) \times L^2(\mathbb{R}^n)$ with compact support
and $N(u)=-|u|^{p-1}u$
(see also \cite{NiZh09}).
He proved that
the asymptotic profile is given by the scaled Gaussian.
However, the asymptotic profile of solutions in higher dimensional cases $n\ge 2$ remains open.
Furthermore, even for the small data global existence,
there are no results for non exponentially decaying initial data when $n\ge 5$.
Here we also refer the reader to
\cite{Kh13, LiNiZh10, LiNiZh11, Wa12, Wa14DCDS}
for space and time dependent damping cases.

In this paper,
we shall prove the existence of the global-in-time solution to the Cauchy problem \eqref{dw}
with suitably small $\varepsilon$
and determine the asymptotic profile. 
Our result extends that of \cite{Ni11AA} to
higher dimensional cases $n\ge 2$,
more general damping $b=b(t)$,
non exponentially decaying initial data 
and with lower order perturbations.
Moreover, in the one-dimensional case, 
we can treat
more general nonlinear terms $N=N(u,u_x,u_t)$
including first order derivatives.
For the proof,
we basically follow the method of Gallay and Raugel \cite{GaRa98}.
To extend their argument to variable damping cases,
we introduce new scaling variables
\[
	s = \log (B(t) + 1 ),\quad y = (B(t) + 1 )^{-1/2}x,\quad
	B(t) = \int_{0}^t \frac{d\tau}{b(\tau)}
\]
instead of \eqref{gr_sv}.
Then, we decompose the solution to the asymptotic profile and the remainder term,
and prove that remainder term decays to zero as time tends to infinity
by using the energy method.
To estimate the energy of the remainder term, in \cite{GaRa98},
they used the primitive of the remainder term
$F(s,y) = \int_{-\infty}^y f(s,z) dz$.
However, this does not work in higher dimensional cases $n\ge 2$.
To overcome this difficulty,
we employ the idea from Coulaud \cite{Co14} in which
asymptotic profiles for the second grade fluids equation
were studied in the three dimensional space.
Namely, we shall use the fractional integral of the form
$\hat{F}(\xi) = |\xi|^{-n/2-\delta}\hat{f}(\xi)$
with $0<\delta<1$,
and apply the energy method to $\hat{F}$ in the Fourier side.
Since the remainder term $f$ satisfies $\hat{f}(0)=0$,
$\hat{F}$ makes sense and
enables us to control the term
$\| \hat{f} \|_{L^2}$
in energy estimates.

This paper is organized as follows.
In the next section, we state the precise assumptions and our main result.
Section 3 is devoted to a proof of the main result.
The proof of energy estimates is divided into the one-dimensional case
and the higher dimensional cases.
After that, we will unify both cases and complete the proof of our result except for the
estimates of the error terms.
These error estimates will be given in Section 4.

We end up this section with some notations used in this paper.
For a complex number $\zeta$, we denote by ${\rm Re\,} \zeta$ its real part.
The letter $C$ indicates a generic positive constant, which may change from line to line.
In particular, we denote by
$C(\ast, \ldots, \ast)$
constants depending on the quantities appearing in parenthesis.
We 
use the symbol
$f \sim g$, which stands for
$C^{-1}g \le f \le C g$
with some $C \ge 1$.
For a function $u=u(t,x):[0,\infty)\times\mathbb{R}^n\to\mathbb{R}$,
we 
write
$u_t=\frac{\partial u}{\partial t}(t)$,
$\partial_{x_i}u=\frac{\partial u}{\partial x_i}\ (i=1,\ldots,n)$,
$\nabla_xu={}^t(\partial_{x_1}u, \ldots, \partial_{x_n}u)$
and
$\Delta u(t,x)=\sum_{i=1}^n \partial_{x_i}^2 u(t,x)$.
Furthermore, 
we sometimes use
$\langle x \rangle := \sqrt{1+|x|^2}$.

For a function $f=f(x):\mathbb{R}^n \to \mathbb{R}$,
we denote the Fourier transform of $f$ by $\hat{f}=\hat{f}(\xi)$, that is,
\[
	\hat{f}(\xi) = (2\pi)^{-n/2} \int_{\mathbb{R}^n} f(x) e^{-ix\xi} dx.
\]
Let $L^p(\mathbb{R}^n)$ and $H^{k,m}(\mathbb{R}^n)$
be usual Lebesgue and weighted Sobolev spaces, respectively,
equipped with the norms defined by
\begin{align*}
	&\| f \|_{L^p} = \left( \int_{\mathbb{R}^n} |f(x)|^p dx \right)^{1/p}\ (1\le p <\infty),\quad
	\| f\|_{L^{\infty}} = {\rm ess\,sup\,}_{x\in \mathbb{R}^n}|f(x)|,\\
	&\| f \|_{H^{k,m}} = \sum_{|\alpha|\le k} \| (1+|x|)^{m} \partial_x^{\alpha}f \|_{L^2}\quad
	(k\in \mathbb{Z}_{\ge 0}, m\ge 0).
\end{align*}
For an interval $I$ and a Banach space $X$,
we define $C^r(I;X)$ as the space of $r$-times continuously differentiable mapping from
$I$ to $X$ with respect to the topology in $X$.


\section{Main result}
Let us introduce our main result.
First, we put the following assumptions:
\vspace{1em}

{\bf Assumptions}
\begin{itemize}
\item[(i)]
The initial data $(u_0,u_1)$ belong to
$H^{1,m}(\mathbb{R}^n) \times H^{0,m}(\mathbb{R}^n)$,
where
$m=1\ (n=1)$
and
$m>n/2+1\ (n\ge 2)$.

\item[(ii)]
The coefficient of the damping term $b(t)$ satisfies
\begin{align}
\label{b0}
	C^{-1}(1+t)^{-\beta} \le b(t) \le C(1+t)^{-\beta},\quad
	\left| \frac{db}{dt}(t) \right| \le C (1+t)^{-1}b(t)
\end{align}
with some $\beta \in [-1,1)$.

\item[(iii)]
The functions $c(t)$ and $d(t)$ satisfy
\begin{align}
\label{c}
	|c(t)| \le C (1+t)^{-\gamma},\quad |d(t)|\le C(1+t)^{-\nu}
\end{align}
with some $\gamma > (1+\beta)/2$ and $\nu > 1+\beta$.

\item[(iv)-(1)]
When $n=1$, the nonlinearity $N$ is of the form
\[
	N = \sum_{i=1}^{k} N_i(u,u_{x},u_{t})
\]
for some $k \ge 0$ and each
$N_i=N_i(z) = N_i(z_1, z_2, z_3)$ satisfies
\begin{align}
\label{N1}
	\left\{\begin{array}{l}
	|N_i(z)|\le C|z_1|^{p_{i1}}|z_2|^{p_{i2}}|z_3|^{p_{i3}},\quad
	p_{ij} \ge 1\ \mbox{or}\ =0,\quad 
	p_{i1} > 1,\quad p_{i2}+p_{i3}\le 1,\\[5pt]
	\displaystyle p_{i1}+2p_{i2}+\left( 3 - \frac{2\beta}{1+\beta} \right)p_{i3} > 3,
	\end{array}\right.
\end{align}
where we note that when $\beta = -1$,
the number
$- 2\beta/(1+\beta)$
is interpreted as an arbitrary large number.
Moreover, to ensure the existence of local-in-time solutions,
we assume that, 
for any $R>0$, there exists a constant $C(R)>0$ such that 
\begin{align}
\label{lip}
	|N_i(z) - N_i(w) |
		\le C(R) \left[ |z_1-w_1|  
			( 1+ |z_2| + |w_2| + |z_3| + |w_3| )
			+ |z_2-w_2| + |z_3 - w_3| \right]
\end{align}
for
$z_i, w_i \in \mathbb{R}\ (i=1,2,3)$
satisfying
$|z_1|, |w_1| \le R$.

\item[(iv)-(2)]
When $n\ge 2$, the nonlinearity $N$
is of class $C^1$
and independent of $\nabla_{x}u, u_{t}$,
that is,
$N=N(u)$.
Moreover, $N$ satisfies
\begin{align}
\label{N2}
	\left\{\begin{array}{l}
	\left| N(u) \right| \le C|u|^{p},\\
	\displaystyle 2 < p <+\infty\ (n=2),\quad 1+\frac{2}{n} < p \le \frac{n}{n-2}\ (n\ge 3).
	\end{array}\right.
\end{align}
Also, to ensure the existence of local-in-time solutions, 
we assume that
\begin{align}
\label{lip2}
		| N(u) - N(v) | \le C |u -v |( |u| + |v| )^{p-1}.
\end{align}
\end{itemize}

\begin{remark}
\begin{itemize}
\item[(i)]By the above assumptions, as we will see later, we can regard
the terms
$c(t)\cdot \nabla_{x}u$, $d(t) u$ and $N(u,\nabla_{x}u,u_{t})$
as perturbations.
\item[(ii)]
We can treat the case where the coefficients
$b(t), c(t)$ and $d(t)$ depend on both $t$ and $x$.
More precisely, our result is also valid for
$b=b(t,x), c=c(t,x)$ and $d=d(t,x)$ such that
$b(t,x) = b_0(t) + b_1(t,x)$
with
$b_0$ satisfying Assumption (ii),
$b_1(t,x)$ fulfilling $|b_1(t,x)|\le C (1+t)^{-\mu}\ (\mu > \beta)$
and $c(t,x), d(t,x)$ satisfying
$|c(t,x)| \le C (1+t)^{-\gamma}\ (\gamma>(1+\beta)/2)$,
$|d(t,x)| \le C (1+t)^{-\nu}\ (\nu >1+\beta)$.
\item[(iii)]
A typical example satisfying the assumptions \eqref{N1} and \eqref{lip} is
\[
	N = |u|^{p}u + |u|^{q}u_x + |u|^{r}u_t
\]
with $p>2$, $q>1$ and $r> 1$. 
\item[(iv)]
The assumption $1+2/n < p$ in \eqref{N2} is sharp in the sense that,
if $N(u) = |u|^p$, $1<p\le 1+2/n$
and the initial data satisfies 
$\int_{\mathbb{R}^n} (u_0+ b_{\ast} u_1)(x) dx > 0$ 
with $b_{\ast} = \int_0^{\infty} \exp ( - \int_0^t b(\tau) d\tau ) dt$, 
then the local-in-time solution blows up in finite time
(see \cite{IkToYo09, LiZh95, LiNiZh11, ToYo01, Zh01}).
\item[(v)]
When $n=1$, we can also treat the principal term with variable coefficient
$(a(x)u_x)_x$ satisfying
\[
	\inf_{x\in\mathbb{R}}a(x) >0,\quad
	\lim_{x\to\pm\infty}a(x) = a_{\pm} > 0
\]
instead of $u_{xx}$.
However, the argument is the same as in Gallay and Raugel \cite{GaRa98}
and hence, we do not pursue here for simplicity.
\item[(vi)]
There are no mutual implication relations
between the assumptions on the damping
$b$ in ours and Wirth \cite{Wi07JDE}, D'Abbicco, Lucente and Reissig \cite{DaLuRe13}.
\end{itemize}
\end{remark}

To state our result, we put
\begin{align}
\label{def_B}
	B(t) = \int_0^{t}\frac{d\tau}{b(\tau)}
\end{align}
and
\begin{align}
\label{gauss}
	\mathcal{G}(t,x) = (4\pi t)^{-n/2} \exp\left( - \frac{|x|^2}{4t} \right).
\end{align}
We note that the assumption \eqref{b0} implies that
$B(t)$ is strictly increasing, and $\lim_{t\to\infty}B(t) = +\infty$.

The main result of this paper is the following:
\begin{theorem}\label{thm}
Under the Assumptions (i)--(iv), there exists some $\varepsilon_0>0$
such that, 
for any $\varepsilon \in (0,\varepsilon_0]$,
there exists a unique solution
\begin{align*}
	u \in C([0,\infty);H^{1,m}(\mathbb{R}^n))\cap C^1([0,\infty);H^{0,m}(\mathbb{R}^n))
\end{align*}
for the Cauchy problem \eqref{dw}.
Moreover,
there exists the limit
\[
	\alpha^{\ast} = \lim_{t\to \infty} \int_{\mathbb{R}^n}u(t,x) \,dx
\]
such that the solution $u$ satisfies
\begin{align}
\label{sol_asym}
	\| u (t,\cdot) - \alpha^{\ast} \mathcal{G}(B(t),\cdot) \|_{L^2}
		\le C \varepsilon (B(t)+1)^{-n/4-\lambda}
			\| (u_0,u_1) \|_{H^{1,m}\times H^{0,m}}
\end{align}
for $t \ge 1$.
Here
$\lambda$
is defined by
\[
	\lambda =
	\min \left\{ \frac{1}{2}, \frac{m}{2}-\frac{n}{4}, \lambda_0,\lambda_1 \right\} -\eta
\]
with arbitrary small number $\eta>0$,
and $\lambda_0$ and $\lambda_1$ are defined by
\[
	\lambda_0 = \min \left\{ \frac{1-\beta}{1+\beta},
						\frac{\gamma}{1+\beta}-\frac{1}{2}, \frac{\nu}{1+\beta}-1 \right\},
\]
where we interpret $1/(1+\beta)$ as an arbitrary large number when $\beta=-1$,
and
\begin{align*}
	\lambda_1 = \begin{cases}
		\displaystyle \frac{1}{2}\min_{i=1,\ldots,k}
			\left\{ p_{i1} +2p_{i2} + \left( 3 - \frac{2\beta}{1+\beta} \right)p_{i3} - 3 \right\},
			& n=1,\\[7pt]
		\displaystyle \frac{n}{2} \left( p - 1- \frac{2}{n} \right), & n\ge 2.
	\end{cases}
\end{align*}
Here we interpret
$-2\beta p_{i3}/(1+\beta)$
as an arbitrary large number when $p_{i3} \neq 0$ and $\beta = -1$.
\end{theorem}
\begin{remark}\label{rem_alpha}
If
$N=c=d=0$, namely there are no perturbation terms,
and if
$\beta$ is close to $1$ so that
$\min \{ 1/2, m/2-n/4, (1-\beta)/(1+\beta)\} = (1-\beta)/(1+\beta)$,
then
$\lambda = (1-\beta)/(1+\beta) - \eta$
with arbitrary small $\eta>0$,
and we expect that
the gain of the decay rate $(1-\beta)/(1+\beta)$ is optimal, in other words,
the second order approximation of $u$
decays as
$(B(t)+1)^{-n/4-(1-\beta)/(1+\beta)}$.
The higher order asymptotic expansion will be discussed in a forthcoming paper.
\end{remark}


\section{Proof of the main theorem}
\subsection{Scaling variables}
We introduce the following scaling variables:
\begin{align}
\label{sc}
	s = \log (B(t) + 1 ),\quad y = (B(t) + 1 )^{-1/2}x
\end{align}
and
\begin{align*}
	v(s,y)=e^{ns/2}u(t(s),e^{s/2}y),\quad
	w(s,y)= b(t(s)) e^{(n+2)s/2}u_t( t(s), e^{s/2}y), 
\end{align*}
or equivalently,
\begin{align}
\label{uvw}
	\begin{array}{l}
	\displaystyle u(t,x) = (B(t)+1)^{-n/2}v(\log(B(t)+1), (B(t)+1)^{-1/2}x),\\[5pt]
	\displaystyle u_t(t,x) = b(t)^{-1}(B(t)+1)^{-n/2-1}w(\log(B(t)+1), (B(t)+1)^{-1/2}x),
	\end{array}
\end{align} 
where we have used the notation 
$t(s) = B^{-1}(e^s-1)$. 
Then, the problem \eqref{dw} is transformed as
\begin{align}
\label{eq_vw}
	\left\{\begin{array}{ll}
	\displaystyle v_s-\frac{y}{2}\cdot \nabla_yv - \frac{n}{2}v = w,&s>0, y\in \mathbb{R}^n,\\[7pt]
	\displaystyle
		\frac{e^{-s}}{b(t(s))^2}
		\left( w_s-\frac{y}{2}\cdot \nabla_yw -\left(\frac{n}{2}+1\right)w \right)+w
		 = \Delta_yv+r(s,y),&s>0,y\in \mathbb{R}^n,\\[7pt]
	\displaystyle v(0,y) = \varepsilon v_0(y) = \varepsilon u_0(y),\ 
				w(0,y) = \varepsilon w_0(y) = \varepsilon b(0) u_1(y),
				&y\in \mathbb{R}^n,
	\end{array}\right.
\end{align}
where
\begin{align}
\nonumber
	r(s,y) &= \frac{1}{b(t(s))^2} \frac{db}{dt}(t(s)) w 
		 +e^{s/2}c(t(s))\cdot\nabla_yv+e^{s}d(t(s))v\\ 
\label{r}
		&\quad +e^{(n+2)s/2}
		 	N\left( e^{-ns/2}v,
						e^{-(n+1)s/2}\nabla_yv,
						b(t(s))^{-1}e^{-(n+2)s/2}w \right).
\end{align}

\subsection{Preliminary lemmas}
First, we collect frequently used relations and estimates.
\begin{lemma}\label{lem_b0}
We have
\begin{align}
\label{d_b}
	\frac{d}{ds} b (t(s)) = \frac{db}{dt}(t(s)) b(t(s)) e^s,
	\quad \frac{d}{ds} \frac{1}{b(t(s))^2} = - \frac{2}{b(t(s))^2} \frac{db}{dt}(t(s))e^s.
\end{align}
\end{lemma}
\begin{proof}
First, we note that the function
$\sigma=B(t)$
is strictly increasing, and hence, the inverse
$t= B^{-1}(\sigma)$
exists and
\[
	\frac{d}{d\sigma} B^{-1}(\sigma) = \left( \frac{dB}{dt}(t) \right)^{-1}
	= b(t).
\]
Combining this with $s=\log (B(t) + 1)$, we obtain
\begin{align*}
	\frac{d}{ds} b (t(s)) &= \frac{d}{ds} b \left( B^{-1}( e^{s} - 1) \right) \\
	&= \frac{db}{dt}(t(s)) \frac{d}{ds} B^{-1}( e^{s} - 1 ) \\
	&= \frac{db}{dt}(t(s)) \left( \frac{dB}{dt} (t(s)) \right)^{-1} \frac{d}{ds}( e^s - 1 )\\
	&= \frac{db}{dt}(t(s)) b(t(s)) e^s.
\end{align*}
This shows the first assertion of \eqref{d_b}.
Moreover, we have
\[
	\frac{d}{ds}\frac{1}{b(t(s))^2}
	= - \frac{2}{b(t(s))^3} \frac{d}{ds}b(t(s))
	= - \frac{2}{b(t(s))^2}  \frac{db}{dt}(t(s)) e^s,
\]
which shows the second assertion of \eqref{d_b}.
\end{proof}

Next, the assumption \eqref{b0} implies the following:
\begin{lemma}\label{lem_b0b1}
Under the assumption \eqref{b0}, we have the following estimates.
\begin{itemize}
\item[(i)]
When $\beta \in (-1,1)$, we have
\begin{align*}
	&b(t(s)) \sim e^{-\beta s/(1+\beta)},\quad
	\frac{e^{-s}}{b(t(s))^2} \sim e^{-(1-\beta)s/(1+\beta)},
	\quad \frac{1}{b(t(s))^2} \left| \frac{db}{dt}(t(s)) \right| \le C e^{-(1-\beta)s/(1+\beta)}.
\end{align*}
\item[(ii)]
When $\beta =-1$, we have
\begin{align*}
	&b(t(s)) \sim \exp \left( e^s \right),\quad
	\frac{e^{-s}}{b(t(s))^2} \sim \exp \left( -2e^s - s \right),\quad
	\frac{1}{b(t(s))^2} \left| \frac{db}{dt}(t(s)) \right|
		\le C \exp \left( -2e^{s} \right).
\end{align*}
\end{itemize}
\end{lemma}
\begin{proof}
(i)
When $\beta \in (-1,1)$, from \eqref{def_B} and \eqref{sc} we compute as
\begin{align*}
	e^s = B(t(s)) + 1 = \int_0^{t(s)}\frac{d\tau}{b(\tau)} +1
		\sim \int_0^{t(s)} (1+\tau)^{\beta} d\tau + 1 \sim (1+t(s))^{1+\beta}.
\end{align*}
Therefore, one has
$1+t(s) \sim e^{s/(1+\beta)}$,
and hence,
\begin{align*}
	b(t(s)) \sim (1+t(s))^{-\beta} \sim e^{-\beta s/(1+\beta)}.
\end{align*}
By the assumption \eqref{b0}, the other estimates can be obtained in a similar way.

(ii)
When $\beta = -1$, we have
\[
	e^{s} = B(t(s)) + 1 \sim  \int_0^{t(s)} (1+\tau)^{-1}d\tau + 1 = \log (1+t(s)) + 1,
\]
and hence, $b(t(s)) \sim 1+t(s) \sim \exp (e^s)$ holds.
We can prove the other estimates in the same way, and the proof is omitted.
\end{proof}

We sometimes employ the
Gagliardo-Nirenberg inequality:
\begin{lemma}[Gagliardo-Nirenberg inequality]\label{lem_gn}
Let $1<p<\infty \ (n=1,2)$ and $1<p\le n/(n-2) \ (n\ge 3)$.
Then for any $f \in H^{1,0}(\mathbb{R}^n)$, we have
\begin{align*}
	\| f \|_{L^{2p}} \le C \| \nabla f \|^{\sigma}_{L^2} \| f \|_{L^2}^{1-\sigma},
\end{align*}
where
$\sigma = n(p-1)/(2p)$.
\end{lemma}
For the proof, see for example \cite{Frbook, GiGiSa}.

\subsection{Local existence of solutions}
We prove the local existence of solutions for
the equation \eqref{dw} and the system \eqref{eq_vw}, respectively.
To this end, putting
$U(t,x) = \langle x \rangle^m u$,
$U_0(x) = \langle x \rangle^m u_0$
and $U_1 = \langle x \rangle^m u_1$,
we change the problem \eqref{dw} to
\begin{align}
\label{dw2}
	\left\{ \begin{array}{ll}
	U_{tt}+b(t)U_{t}
	=\Delta_{x}U + \tilde{c}(t,x)\cdot \nabla_{x}U + \tilde{d}(t,x)U
		+ \tilde{N}(U,\nabla_{x}U,U_{t}),&
	t>0,x\in\mathbb{R}^n,\\
	U(0,x)=\varepsilon U_0(x),
		\ U_{t}(0,x)=\varepsilon U_1(x),
	&x \in\mathbb{R}^n,
	\end{array}\right.
\end{align}
where
$\tilde{c} = c - 2m \langle x \rangle^{-2}x$,
$\tilde{d} = d - c\cdot (m \langle x \rangle^{-2} x )
- m \langle x \rangle^{-4} (n\langle x \rangle^{2} - (m+2)|x|^2)$
and
\[
	\tilde{N}(U,\nabla_xU, U_t)
	= \langle x \rangle^m
	N \left( \langle x \rangle^{-m}U,
		\langle x\rangle^{-m} \nabla_x U - m\langle x \rangle^{-m-2}x U,
		\langle x \rangle^{-m}U_t \right).
\]
We further put
$\mathcal{U} = {}^{t}(U, U_t)$ and $\mathcal{U}_0 = {}^t (U_0, U_1)$.
Then, the equation \eqref{dw2} is written as
\begin{align}
\label{dw3}
	\left\{ \begin{array}{l}
	\mathcal{U}_t = \mathcal{A}\,\mathcal{U} + \mathcal{N}(\,\mathcal{U}),\\
	\mathcal{U}(0) = \varepsilon\,\mathcal{U}_0,
	\end{array}\right.
\end{align}
where
\[
	\mathcal{A} =
		\left( \begin{array}{cc} 0 & 1\\ \Delta & 0 \end{array} \right),\quad
	\mathcal{N}(\,\mathcal{U})
		= \left( \begin{array}{cc}
			0\\
			-b U_t + \tilde{c}\cdot \nabla_x U + \tilde{d} U + \tilde{N}(U, \nabla_x U, U_t)
			\end{array} \right).
\]
The operator
$\mathcal{A}$ on $H^{1,0}(\mathbb{R}^n) \times L^2(\mathbb{R}^n)$
with the domain
$D (\mathcal{A}) = H^{2,0}(\mathbb{R}^n) \times H^{1,0}(\mathbb{R}^n)$
is $m$-dissipative (see \cite[Proposition 2.6.9]{CaHa}) with dense domain,
and hence, $\mathcal{A}$ generates a contraction semigroup
$e^{t \mathcal{A}}$ on $H^{1,0}(\mathbb{R}^n) \times L^2(\mathbb{R}^n)$
(see \cite[Theorem 3.4.4]{CaHa}).
Thus, we consider the integral form
\begin{align}
\label{mildsol}
	\mathcal{U}(t) = \varepsilon e^{t\mathcal{A}} \mathcal{U}_0
		+ \int_0^t e^{(t-\tau)\mathcal{A}} \mathcal{N}(\,\mathcal{U}(\tau))\,d\tau
\end{align}
of the equation \eqref{dw3}
in
$C([0,T); H^{1,0}(\mathbb{R}^n) \times L^2(\mathbb{R}^n))$.

First, we define the mild and strong solutions and
the lifespan of solutions.

\begin{definition}
We say that $u$ is a mild solution of the Cauchy problem \eqref{dw}
on the interval
$[0,T)$
if
$u$ has the regularity
\begin{align}
\label{reg0}
	u \in C([0,T);H^{1,m}(\mathbb{R}^n))\cap C^1([0,T);H^{0,m}(\mathbb{R}^n)).
\end{align}
and
satisfies the integral equation \eqref{mildsol}
in $C([0,T); H^{1,0}(\mathbb{R}^n)\times L^2(\mathbb{R}^n))$.
We also call $u$ a strong solution of the Cauchy problem \eqref{dw}
on the interval $[0,T)$
if
$u$ has the regularity
\begin{align}
\label{reg}
	u \in C([0,T);H^{2,m}(\mathbb{R}^n))\cap C^1([0,T);H^{1,m}(\mathbb{R}^n))
		\cap C^2([0,T); H^{0,m}(\mathbb{R}^n))
\end{align}
and
satisfies the equation \eqref{dw} in
$C([0,T);H^{0,m}(\mathbb{R}^n))$.
Moreover, we say that
$(v,w)$ defined by \eqref{uvw}
is a mild (resp. strong) solution of the Cauchy problem \eqref{eq_vw}
on the interval $[0,S)$
if $u$ is a mild (resp. strong) solution of \eqref{dw}
on the interval $[0, t(S))$.
We note that if
$(v,w)$ is a mild solution of \eqref{eq_vw} on $[0,S)$,
then
$(v,w)$
has the regularity
\[
	(v,w) \in C([0,S);H^{1,m}(\mathbb{R}^n)\times H^{0,m}(\mathbb{R}^n)),
\]
and if $(v,w)$ is a strong solution of \eqref{eq_vw} on $[0,S)$,
then
$(v,w)$
has the regularity
\begin{align}
\label{vwcls2}
	(v,w) \in C([0,S);H^{2,m}(\mathbb{R}^n)\times H^{1,m}(\mathbb{R}^n))
		\cap C^1([0,S); H^{1,m}(\mathbb{R}^n) \times H^{0,m}(\mathbb{R}^n))
\end{align}
and satisfies
the system \eqref{eq_vw} in
$C([0,S);H^{1,m}(\mathbb{R}^n)\times H^{0,m}(\mathbb{R}^n))$.

We also define the lifespan of the mild solutions
$u$ and $(v,w)$
by
\[
	T(\varepsilon) = \sup\{
		T\in (0,\infty) ; \mbox{there exists a unique  mild solution}\ u\ \mbox{to \eqref{dw}} \}
\]
and
\[
	S(\varepsilon) = \sup\{
		S\in (0,\infty) ;
		\mbox{there exists a unique  mild solution}\ (v,w)\ \mbox{to \eqref{eq_vw}} \},
\]
respectively.
\end{definition}

\begin{proposition}\label{prop_loc}
Under the assumptions (i)--(iv) in the previous section,
there exists $T>0$ depending only on
$\varepsilon \| (u_0, u_1) \|_{H^{1,m}\times H^{0,m}}$
(the size of the initial data) such that
the Cauchy problem \eqref{dw} admits a unique mild solution
$u$.
Also, if
$(u_0, u_1) \in H^{2,m}(\mathbb{R}^n) \times H^{1,m}(\mathbb{R}^n)$
in addition to Assumption (i), then the corresponding mild solution $u$
becomes a strong solution of \eqref{dw}.
Moreover,
if the lifespan
$T(\varepsilon)$
is finite, then $u$ satisfies
$\lim_{t \to T(\varepsilon)} \| (u, u_t)(t) \|_{H^{1,m}\times H^{0,m}} = \infty$.
Furthermore,
for arbitrary fixed time $T_0>0$,
we can extend the solution to the interval $[0,T_0)$
by taking $\varepsilon$ sufficiently small.
\end{proposition}
From this proposition, we easily have the following.
\begin{proposition}\label{prop_loc2}
Under the assumptions (i)--(iv) in the previous section,
there exists $S>0$ depending only on
$\varepsilon \| (v_0, w_0) \|_{H^{1,m}\times H^{0,m}}$
(the size of the initial data) such that
the Cauchy problem \eqref{eq_vw} admits a unique mild solution
$(v,w)$.
Also, if
$(u_0, u_1) \in H^{2,m}(\mathbb{R}^n) \times H^{1,m}(\mathbb{R}^n)$
in addition to Assumption (i), then the corresponding mild solution $(v,w)$
becomes a strong solution of \eqref{eq_vw}.
Moreover,
if the lifespan
$S(\varepsilon)$
is finite, then $(v,w)$ satisfies
$\lim_{s \to S(\varepsilon)} \| (v,w)(s) \|_{H^{1,m}\times H^{0,m}} = \infty$.
Furthermore,
for arbitrary fixed time $S_0>0$,
we can extend the solution to the interval $[0,S_0]$
by taking $\varepsilon$ sufficiently small.
\end{proposition}
\begin{proof}[Proof of Proposition \ref{prop_loc}]
By using the assumption (iv), and 
the Sobolev inequality for $n=1$, or
the Gagliardo-Nirenberg inequality for $n\ge 2$
(see Lemma \ref{lem_gn}),
we can see that
$\mathcal{N}(\,\mathcal{U})$
is a locally Lipschitz mapping on $H^{1,0}(\mathbb{R}^n) \times L^2(\mathbb{R}^n)$.
Therefore, by \cite[Proposition 4.3.3]{CaHa},
there exists a unique solution
$\mathcal{U} \in C([0,T) ; H^{1,0}(\mathbb{R}^n) \times L^2(\mathbb{R}^n))$
to the integral equation \eqref{mildsol}.
This shows the existence of a unique mild solution
$u$
to the Cauchy problem \eqref{dw}.

If
$(u_0, u_1)\in H^{2,m}(\mathbb{R}^n) \times H^{1,m}(\mathbb{R}^n)$,
then we have
$\mathcal{U}_0 \in D(\mathcal{A})$,
and hence, \cite[Proposition 4.3.9]{CaHa} implies that
$\mathcal{U} \in C([0,T); D(\mathcal{A}))
\cap C^1([0,T); H^{1,0}(\mathbb{R}^n) \times L^2(\mathbb{R}^n))$
and
$\mathcal{U}$
becomes the strong solution of the equation \eqref{dw3},
namely, $\mathcal{U}$ satisfies the equation \eqref{dw3}
in $C([0,T);H^{1,0}(\mathbb{R}^n) \times L^2(\mathbb{R}^n))$.
Then, by the definition of $\mathcal{U}$,
we conclude that $u$ has the regularity in \eqref{reg} and
satisfies the equation \eqref{dw} in
$C([0,T); H^{0,m}(\mathbb{R}^n))$.
Moreover, employing
\cite[Theorem 4.3.4]{CaHa},
we see that
if the lifespan
$T(\varepsilon)$
is finite, then $u$ satisfies
$\lim_{t \to T(\varepsilon)} \| (u, u_t)(t) \|_{H^{1,m}\times H^{0,m}} = \infty$.

Next, we prove that for any fixed $T_0 > 0$,
the solution $u$ can be extended over the interval $[0,T_0]$
by taking $\varepsilon$ sufficiently small.
To verify this, we reconsider the Cauchy problem \eqref{dw2} and its
inhomogeneous linear version
\begin{align}
\label{dw4}
	\left\{ \begin{array}{ll}
	U_{tt}+b(t)U_{t}
	=\Delta_{x}U + \tilde{c}(t,x)\cdot \nabla_{x}U + \tilde{d}(t,x)U
		+ \tilde{N}(t,x),&
	t>0,x\in\mathbb{R}^n,\\
	U(0,x)=\varepsilon U_0(x),
		\ U_{t}(0,x)=\varepsilon U_1(x),
	&x \in\mathbb{R}^n.
	\end{array}\right.
\end{align}
For $\tilde{N} \in L^1(0,T_0;L^2(\mathbb{R}^n))$,
the existence of a unique solution in the distribution sense
is proved by \cite[Theorem 23.2.2]{Ho}.
We also recall the standard energy estimate (see \cite[Lemma 23.2.1]{Ho})
\begin{align}
\label{enes1}
	\sup_{0 < t <T_0} \| (U,U_t)(t) \|_{H^{1,0}\times L^2}
	\le C(T_0) \left(
		\varepsilon \| (U_0, U_1) \|_{H^{1,0}\times L^2}
		+ \int_0^{T_0} \| \tilde{N}(t) \|_{L^2} dt \right).
\end{align}
We again construct the solution $U$ to \eqref{dw2} in
\[
	K:=
	\left\{ U \in C([0,T_0];H^{1,0}(\mathbb{R}^n)) \cap C^1([0,T_0];L^2(\mathbb{R}^n)) ;
		\sup_{0<t<T_0} \| (U,U_t)(t) \|_{H^{1,0}\times L^2} \le 2 C(T_0) I_0 \varepsilon \right\},
\]
where
$I_0 := \| (U_0, U_1) \|_{H^{1,0} \times L^2}$.
For each $V \in K$,
we define the mapping by $U= \mathcal{M}(V)$,
where $U$ is the solution to \eqref{dw4} with
$\tilde{N}= \tilde{N}(V,\nabla_xV, V_t)$.
Then, by using the Sobolev inequality or the Gagliardo-Nirenberg inequality again
with the estimate \eqref{enes1},
we can see that
\begin{align}
\label{estU0}
	\sup_{0<t<T_0}\| (U,U_t)(t) \|_{H^{1,0}\times L^2}
		\le C(T_0) I_0 \varepsilon
		+ C(T_0) ( 2 C(T_0) I_0 \varepsilon )^p T_0.
\end{align}
Thus, noting $p>1$ and taking $\varepsilon > 0$ sufficiently small, we deduce that
$\mathcal{M}$ maps $K$ to itself.
Furthermore, in the same manner, we easily obtain
\begin{align*}
	\sup_{0<t<T_0}\| (U^1,U^1_t) - (U^2, U^2_t) \|_{H^{1,0}\times L^1}
	&\le C(T_0) (4 C(T_0) I_0 \varepsilon)^{p-1}
			T_0 \sup_{0<t<T_0}\| (V^1,V^1_t) - (V^2, V^2_t) \|_{H^{1,0}\times L^2},
\end{align*}
where $U^j = \mathcal{M}(V_j)\ (j=1,2)$.
Thus, noting $p >1$ again and taking $\varepsilon$ sufficiently small, we see that
$\mathcal{M}$ is a contraction mapping on $K$.
Therefore, by the contraction mapping principle,
we find a unique fixed point $\tilde{U}$ of the mapping $\mathcal{M}$
in the set $K$, and
$\tilde{U}$ satisfies the equation \eqref{dw2} in the distribution sense.
Also, the uniqueness of the solution in the distribution sense to \eqref{dw2}
in the class
$C([0,T_0];H^{1,0}(\mathbb{R}^n)) \cap C^1([0,T_0];L^2(\mathbb{R}^n))$
follows from \eqref{enes1}.
Since the mild solution $U$ constructed before also satisfies the equation
\eqref{dw2} in the distribution sense,
we have $U(t) =\tilde{U}(t)$ for $t\in [0, \min\{ T(\varepsilon), T_0 \} )$.
However, noting that
the estimate \eqref{estU0} implies
$\sup_{0<t<T_0}\| (\tilde{U}, \tilde{U}_t)(t) \|_{H^{1,0}\times L^2}$
is finite, we have $T_0 < T(\varepsilon)$
and this completes the proof.
\end{proof}


\subsection{A priori estimate implies the global existence}
In what follows,
to justify the energy method,
we tacitly assume that
$(u_0, u_1) \in H^{2,m}(\mathbb{R}^n)\times H^{1,m}(\mathbb{R}^n)$,
and the solution
$(v,w)$ is in the class \eqref{vwcls2}.
Therefore, the following calculations make sense.
Once we obtain the desired asymptotic estimate \eqref{sol_asym}
for such a data, we can easily have the same estimate
for general $(u_0, u_1) \in H^{1,m}(\mathbb{R}^n)\times H^{0,m}(\mathbb{R}^n)$
by applying the usual approximation argument.

Let
$(v,w)$
be the local-in-time solution to \eqref{eq_vw} on the interval
$[0,S)$.
By the local existence theorem,
it suffices to show an a priori estimate of solutions.
The first goal of this section is the following a priori estimate:
\begin{proposition}\label{prop_ap}
Under the assumptions (i)--(iv) in the previous section,
there exist constants
$s_0 > 0$,
$\varepsilon_1>0$
and
$C_{\ast}>0$ such that the following holds:
if
$\varepsilon \in (0,\varepsilon_1]$
and
$(v,w)$ is a mild solution of \eqref{dw} on some interval
$[0,S]$ with $S> s_0$,
then $(v,w)$ satisfies
\begin{align}
\label{apri}
	\| v(s) \|_{H^{1,m}}^2 + \frac{e^{-s}}{b(t(s))^2} \| w(s) \|_{H^{0,m}}^2 \le
		C_{\ast} \varepsilon^2 \| (v_0, w_0) \|_{H^{1,m}\times H^{0,m}}^2.
\end{align}
\end{proposition}

Before proving the above proposition, we show that
Propositions \ref{prop_loc2} and \ref{prop_ap} imply
the global existence of solutions for small $\varepsilon$.
\begin{proof}[Proof of global existence part of Theorem \ref{thm}] 
First, we note that
Proposition \ref{prop_loc2} guarantees that
there exits $\varepsilon_2 > 0$ such that
the mild solution $(v,w)$ uniquely exists on the interval
$[0,s_0]$ for
$\varepsilon \in (0,\varepsilon_2]$,
where $s_0$ is the constant described in Proposition \ref{prop_ap}.
In particular, we have
$S(\varepsilon) > s_0$
for $\varepsilon \in (0,\varepsilon_2]$.
Let $\varepsilon_0 := \min\{ \varepsilon_1, \varepsilon_2 \}$,
where $\varepsilon_1$ is the constant described in Proposition \ref{prop_ap}.
Then, we have $S(\varepsilon) = \infty$ for $\varepsilon \in (0, \varepsilon_0]$.
Indeed, suppose that
$S(\varepsilon_{\ast}) < \infty$
for some $\varepsilon_{\ast} \in (0, \varepsilon_0]$
and let
$(v,w)$ be the corresponding mild solution of \eqref{eq_vw}.
Applying Proposition \ref{prop_ap}, we have
the a priori estimate \eqref{apri} with $\varepsilon = \varepsilon_{\ast}$.
On the other hand, Proposition \ref{prop_loc2} also implies
\[
	\lim_{s \to S(\varepsilon_{\ast}) } \| (v,w)(s) \|_{H^{1,m}\times H^{0,m}} = \infty.
\]
However, it contradicts the a priori estimate \eqref{apri}.
Thus, we have $S(\varepsilon)=\infty$ for $\varepsilon \in (0,\varepsilon_0]$.
\end{proof}

\subsection{Spectral decomposition}
In the following, we prove the a priori estimate \eqref{apri}
in Proposition \ref{prop_ap}.
At first, we decompose
$v$ and $w$ into the leading terms and the remainder terms, respectively.

Let $\alpha(s)$ be
\begin{align}
\label{alpha}
	\alpha(s) = \int_{\mathbb{R}^n}v(s,y)dy.
\end{align}
Since $v(s) \in H^{1, m}(\mathbb{R}^n)$ for each $s\in [0,S)$ and $m> n/2$,
$\alpha(s)$ is well-defined.
We also put
\begin{align*}
	\varphi_0(y) = (4\pi)^{-n/2} \exp \left( -\frac{|y|^2}{4} \right).
\end{align*}
Then, it is easily verified that
\begin{align}
\label{varphi0_int}
	\int_{\mathbb{R}^n}\varphi_0(y) dy = 1
\end{align}
and
\begin{align}
\label{phi_eq}
	\Delta \varphi_0 = -\frac{y}{2}\cdot \nabla_y\varphi_0-\frac{n}{2}\varphi_0.
\end{align}
We also put
\begin{align*}
	\psi_0(y) = \Delta \varphi_0(y).
\end{align*}
We decompose $v, w$ as
\begin{align}
\label{sp_de_vw}
	\begin{array}{l}
	\displaystyle v(s,y) = \alpha(s) \varphi_0(y) + f(s,y),\\[5pt]
	\displaystyle w(s,y) = \frac{d\alpha}{ds}(s) \varphi_0(y) + \alpha(s)\psi_0(y) + g(s,y).
	\end{array}
\end{align}
We shall prove that $f,g$ can be regarded as remainder terms.

First, we note the following lemma.
\begin{lemma}\label{lem_alpha}
We have
\begin{align}
\label{alpha_dt}
	\frac{d\alpha}{ds}(s) &= \int_{\mathbb{R}^n}w(s,y)dy,\\
\label{alpha_ddt}
	\frac{e^{-s}}{b(t(s))^2}\frac{d^2\alpha}{ds^2}(s)
	&= \frac{e^{-s}}{b(t(s))^2} \frac{d\alpha}{ds}(s)
		- \frac{d\alpha}{ds}(s) + \int_{\mathbb{R}^n}r(s,y)dy,
\end{align}
where $r$ is defined by \eqref{r}.
\end{lemma}
\begin{proof}
Noting
$v \in C^1 ([0,S);H^{1,m}(\mathbb{R}^n))$, $w \in C([0,S);H^{0,m}(\mathbb{R}^n))$ and $m>n/2$,
we immediately obtain \eqref{alpha_dt} from
\begin{align*}
	\frac{d\alpha}{ds}(s)= \int_{\mathbb{R}^n} v_s(s,y)dy
	= \int_{\mathbb{R}^n} \left( \frac{y}{2}\cdot \nabla_y v + \frac{n}{2}v + w \right) dy
	= \int_{\mathbb{R}^n} w (s,y)dy.
\end{align*}
Next, by the regularity \eqref{vwcls2},
we see that
$\frac{d\alpha}{ds}(s) \in C^1([0,S); \mathbb{R})$.
Differentiating
$\frac{d\alpha}{ds}(s)$
again and using the second equation of \eqref{eq_vw}, we have 
\begin{align*}
	\frac{e^{-s}}{b(t(s))^2}\frac{d^2\alpha}{ds^2}(s)
	&= \frac{e^{-s}}{b(t(s))^2}\int_{\mathbb{R}^n}w_s (s,y)dy\\
	&= \frac{e^{-s}}{b(t(s))^2}\int_{\mathbb{R}^n}
		\left( \frac{y}{2}\cdot \nabla_yw+\left(\frac{n}{2}+1\right)w \right) dy
		- \int_{\mathbb{R}^n}wdy + \int_{\mathbb{R}^n}\Delta_y v dy + \int_{\mathbb{R}^n}r dy\\
	&= \frac{e^{-s}}{b(t(s))^2}\int_{\mathbb{R}^n}w dy
		- \int_{\mathbb{R}^n}wdy+ \int_{\mathbb{R}^n}r dy.
\end{align*}
Thus, we finish the proof.
\end{proof}

Next, we consider the remainder term $(f,g)$.
Since $f$ and $g$ are defined by \eqref{sp_de_vw}, and
we assumed that
$(v,w)$ has the regularity in \eqref{vwcls2},
so is $(f,g)$:
\begin{align}
\label{fgcls2}
	(f,g) \in C([0,S);H^{2,m}(\mathbb{R}^n)\times H^{1,m}(\mathbb{R}^n))
		\cap C^1([0,S); H^{1,m}(\mathbb{R}^n) \times H^{0,m}(\mathbb{R}^n)).
\end{align}
Therefore, from the system \eqref{eq_vw} and the equation \eqref{phi_eq},
we see that $f$ and $g$ satisfy the following system:
\begin{align}
\label{eq_fg}
	\left\{\begin{array}{ll}
	\displaystyle
	f_s - \frac{y}{2}\cdot \nabla_yf-\frac{n}{2}f = g,&s>0, y\in\mathbb{R}^n,\\[5pt]
	\displaystyle
	\frac{e^{-s}}{b(t(s))^2}\left( g_s - \frac{y}{2}\cdot\nabla_y g -\left(\frac{n}{2}+1\right) g \right)
		+ g = \Delta_y f + h,&s>0,y\in\mathbb{R}^n,\\[5pt]
	f(0,y) = v(0,y)-\alpha(0)\varphi_0(y),&y\in\mathbb{R}^n,\\[5pt]
	g(0,y) = w(0,y)-\dot{\alpha}(0)\varphi_0(y)-\alpha(0)\psi_0(y),
		&y\in\mathbb{R}^n,
	\end{array}\right.
\end{align}
where $h$ is given by
\begin{align}
\nonumber
	h(s,y) &= \frac{e^{-s}}{b(t(s))^2}
		\left( -2 \frac{d\alpha}{ds}(s) \psi_0(y)
		+\alpha(s)
		\left(\frac{y}{2}\cdot\nabla_y\psi_0(y)
			+\left(\frac{n}{2}+1\right)\psi_0(y) \right) \right)\\
\label{h}
		&\quad + r(s,y) - \left(\int_{\mathbb{R}^n} r(s,y) dy \right) \varphi_0(y).
\end{align}
Moreover, from \eqref{alpha}, \eqref{varphi0_int} and \eqref{alpha_dt}, it follows that
\begin{align}
\label{fg_int}
	\int_{\mathbb{R}^n}f(s,y)dy = \int_{\mathbb{R}^n}g(s,y)dy = 0.
\end{align}
We also notice that the condition \eqref{fg_int} implies
\begin{align}
\label{h_int}
	\int_{\mathbb{R}^n} h(s,y) dy = 0.
\end{align}
We note that it suffices to show a priori estimates of
$f$, $g$, $\alpha$ and $\frac{d\alpha}{ds}$
for the proof of global existence of solutions to the system \eqref{eq_vw}.
Therefore, hereafter, we consider the system \eqref{eq_fg} instead of \eqref{eq_vw}.

\subsection{Energy estimates for $n=1$}
To obtain the decay estimates for $f,g$, we introduce
\begin{align}
\label{1_FG}
	F(s,y) = \int_{-\infty}^y f(s,z)dz,\quad
	G(s,y) = \int_{-\infty}^y g(s,z)dz.
\end{align}
From the following lemma and the condition \eqref{fg_int},
we see that $F,G\in C([0,S); L^2(\mathbb{R}))$. 

\begin{lemma}[Hardy-type inequality]\label{lem_hardy}
Let $f=f(y)$ belong to $H^{0,1}(\mathbb{R})$ and satisfy
$\int_{\mathbb{R}}f(y)dy = 0$,
and let $F(y) = \int_{-\infty}^y f(z)dz$.
Then it holds that
\begin{align}
\label{hardy}
	\int_{\mathbb{R}}F(y)^2 dy \le 4 \int_{\mathbb{R}}y^2 f(y)^2 dy.
\end{align}
\end{lemma}
\begin{proof}
First, we prove \eqref{hardy} when $f\in C_0^{\infty}(\mathbb{R})$.
In this case $\int_{\mathbb{R}}f(y)dy = 0$ leads to
$F\in C_0^{\infty}(\mathbb{R})$.
Therefore, we apply the integration by parts and have
\begin{align*}
	\int_{\mathbb{R}}F(y)^2 dy
	= -2\int_{\mathbb{R}}yF(y)f(y)dy
	\le 2\int_{\mathbb{R}} y^2 f(y)^2 dy + \frac{1}{2}\int_{\mathbb{R}} F(y)^2 dy.
\end{align*}
Thus, we obtain \eqref{hardy}.
For general $f\in H^{0,1}(\mathbb{R})$ satisfying $\int_{\mathbb{R}}f(y)dy=0$,
we can easily prove \eqref{hardy} by appropriately approximations.
\end{proof}

Moreover, by the regularity assumption \eqref{fgcls2} on $(f,g)$,
we see that
\begin{align}
\label{cfcgcls2n1}
	(F,G) \in C([0,S);H^{3,0}(\mathbb{R})\times H^{2,0}(\mathbb{R}))
		\cap C^1([0,S); H^{2,0}(\mathbb{R}) \times H^{1,0}(\mathbb{R})).
\end{align}
Since $f$ and $g$ satisfy the equation \eqref{eq_fg}, we can show that
$F$ and $G$ satisfy the following system:
\begin{align}
\label{eq_FG}
	\left\{\begin{array}{ll}
	\displaystyle F_s-\frac{y}{2}F_y = G,&s>0,y\in\mathbb{R},\\[5pt]
	\displaystyle
	\frac{e^{-s}}{b(t(s))^2}\left( G_s - \frac{y}{2}G_y -G \right) + G
	= F_{yy} + H,&s>0, y\in \mathbb{R},\\[5pt] 
	\displaystyle F(0,y) = \int_{-\infty}^{y}f(0,z)dz,\ 
	G(0,y) = \int_{-\infty}^{y}g(0,z)dz, &y\in \mathbb{R},
	\end{array}\right.
\end{align}
where
\begin{align}
\label{H}
	H(s,y) = \int_{-\infty}^yh(s,z)dz.
\end{align}

We define the following energy.
\begin{align*}
	E_0(s) &= \int_{\mathbb{R}} \left( \frac{1}{2}\left( F_y^2 + \frac{e^{-s}}{b(t(s))^2}G^2 \right)
		+ \frac{1}{2}F^2 + \frac{e^{-s}}{b(t(s))^2} FG \right) dy,\\
	E_1(s) &= \int_{\mathbb{R}} \left( \frac{1}{2} \left( f_y^2 + \frac{e^{-s}}{b(t(s))^2}g^2 \right)
		+ f^2 + 2\frac{e^{-s}}{b(t(s))^2}fg \right)dy,\\
	E_2(s) & = \int_{\mathbb{R}} y^2 \left[ \frac{1}{2} \left( f_y^2 + \frac{e^{-s}}{b(t(s))^2}g^2 \right)
		+ \frac{1}{2} f^2 + \frac{e^{-s}}{b(t(s))^2}fg \right] dy.
\end{align*}
By using Lemma \ref{lem_b0b1},
the following equivalents are valid for
$s\ge s_1$
with sufficiently large
$s_1>0$.
\begin{align}
\nonumber
	E_0(s) & \sim \int_{\mathbb{R}}
		\left( F_y^2 + \frac{e^{-s}}{b(t(s))^2} G^2 + F^2 \right) dy,\\
\label{equi1}
	E_1(s) & \sim \int_{\mathbb{R}}
		\left( f_y^2 + \frac{e^{-s}}{b(t(s))^2} g^2 + f^2 \right) dy,\\
\nonumber
	E_2(s) &\sim \int_{\mathbb{R}}
		y^{2} \left[ f_y^2 + \frac{e^{-s}}{b(t(s))^2} g^2 + f^2 \right] dy.
\end{align}

Next, we prove the following energy identity.
\begin{lemma}\label{lem_en0}
We have
\begin{align}
\label{e0n1}
	\frac{d}{ds}E_0(s)
	+\frac{1}{2}E_0(s) + L_0 (s)
	= R_0 (s),
\end{align}
where
\begin{align*}
	L_0(s) &= \int_{\mathbb{R}}\left( \frac{1}{2}F_y^2 + G^2 \right)dy,\\
	R_0(s) &= \frac{3}{2}\frac{e^{-s}}{b(t(s))^2}\int_{\mathbb{R}}G^2 dy 
		- \frac{1}{b(t(s))^2} \frac{db}{dt}(t(s))
				\int_{\mathbb{R}}\left( G^2 + 2FG \right)dy
		+ \int_{\mathbb{R}} (F+G)H dy.
\end{align*}
Moreover, we have
\begin{align}
\label{e1n1}
	\frac{d}{ds}E_1(s) + \frac{1}{2}E_1(s) + L_1(s) = R_1(s),
\end{align}
where
\begin{align*}
	L_1(s) &= \int_{\mathbb{R}}\left( f_y^2 + g^2 \right) dy
	- \int_{\mathbb{R}}f^2 dy,\\
	R_1(s) &= 3\frac{e^{-s}}{b(t(s))^2}\int_{\mathbb{R}}g^2 dy
	+ 2\frac{e^{-s}}{b(t(s))^2}\int_{\mathbb{R}}fg dy
	- \frac{1}{b(t(s))^2} \frac{db}{dt}(t(s)) \int_{\mathbb{R}}(g^2+4fg)dy
	+ \int_{\mathbb{R}} \left(2f+g\right) h dy.
\end{align*}
Furthermore, we have 
\begin{align}
\label{e2n1}
	\frac{d}{ds}E_2(s) + \frac{1}{2}E_2(s)
		+ L_2(s) =R_2(s),
\end{align}
where
\begin{align*}
	L_2(s) &= \int_{\mathbb{R}}y^2 \left( \frac{1}{2}f_y^2 + g^2 \right)dy
		+ 2\int_{\mathbb{R}}y f_y \left( f+ g \right) dy,\\
	R_2(s) &= \frac{3}{2}\frac{e^{-s}}{b(t(s))^2}\int_{\mathbb{R}}y^2  g^2 dy
		-\frac{1}{b(t(s))^2} \frac{db}{dt}(t(s)) \int_{\mathbb{R}}y^2 (2f+g)g dy
		+\int_{\mathbb{R}}y^2 (f+g)hdy.
\end{align*}
\end{lemma}%
\begin{proof}%
The proofs of \eqref{e1n1} and \eqref{e2n1} are the almost same as
that of \eqref{e0n1},
and we only prove \eqref{e0n1}.
We calculate the derivatives of each term of $E_0(s)$.
First, we have
\begin{align*}
	\frac{d}{ds} \left[ \frac{1}{2} \int_{\mathbb{R}} F^2 dy \right] 
	&= \int_{\mathbb{R}} F F_s dy \\
	&= \int_{\mathbb{R}} F \left( \frac{y}{2}F_y + G \right) dy\\
	&= \int_{\mathbb{R}} \left( \left( \frac{y}{4}F^2 \right)_y - \frac{1}{4}F^2 +FG \right) dy \\
	&= -\frac{1}{4} \int_{\mathbb{R}} F^2 dy + \int_{\mathbb{R}} FG dy.
\end{align*}
Here we have used that
$\frac{y}{2} F F_y \in L^1(\mathbb{R})$,
which enables us to justify the integration by parts.
By Lemma \ref{lem_b0}, we also have
\begin{align*}
	\frac{d}{ds}\left[ \frac{e^{-s}}{b(t(s))^2} \int_{\mathbb{R}} FG dy \right]
	&= -\frac{2}{b(t(s))^2} \frac{db}{dt}(t(s)) \int_{\mathbb{R}}FGdy
		-\frac{e^{-s}}{b(t(s))^2}\int_{\mathbb{R}}FGdy
		+ \frac{e^{-s}}{b(t(s))^2}\int_{\mathbb{R}}(F_sG+FG_s)dy\\
	&= -\frac{2}{b(t(s))^2} \frac{db}{dt}(t(s)) \int_{\mathbb{R}}FGdy
		-\frac{e^{-s}}{b(t(s))^2}\int_{\mathbb{R}}FGdy
		+\frac{e^{-s}}{b(t(s))^2}\int_{\mathbb{R}} \left( \frac{y}{2}F_y + G \right) G dy\\
	&\quad +\frac{e^{-s}}{b(t(s))^2}\int_{\mathbb{R}} F\left( \frac{y}{2}G_y + G \right) dy
		-\int_{\mathbb{R}}FGdy + \int_{\mathbb{R}}FF_{yy}dy + \int_{\mathbb{R}}FHdy\\ 
	&= -\frac{1}{2}\frac{e^{-s}}{b(t(s))^2}\int_{\mathbb{R}}FGdy
		-\frac{2}{b(t(s))^2} \frac{db}{dt}(t(s)) \int_{\mathbb{R}}FGdy\\
	&\quad + \frac{e^{-s}}{b(t(s))^2}\int_{\mathbb{R}}G^2 dy - \int_{\mathbb{R}}FGdy 
		-\int_{\mathbb{R}}F_y^2dy + \int_{\mathbb{R}}FHdy. 
\end{align*}
Adding up the above identities, we conclude that
\begin{align}
\nonumber
	\frac{d}{ds}\left[
		\int_{\mathbb{R}}\left( \frac{1}{2}F^2 + \frac{e^{-s}}{b(t(s))^2}FG \right)dy
			\right]
	&= -\frac{1}{4}\int_{\mathbb{R}}F^2dy
		- \frac{1}{2}\frac{e^{-s}}{b(t(s))^2}\int_{\mathbb{R}} FG dy
		- \frac{2}{b(t(s))^2} \frac{db}{dt}(t(s)) \int_{\mathbb{R}}FGdy\\
\label{en1}
	&\quad +\frac{e^{-s}}{b(t(s))^2}\int_{\mathbb{R}}G^2dy - \int_{\mathbb{R}}F_y^2 dy
		+ \int_{\mathbb{R}}FH dy. 
\end{align}
We also have
\begin{align*}
	\frac{d}{ds}\left[ \frac{1}{2} \int_{\mathbb{R}} F_y^2 dy \right]
	&= \int_{\mathbb{R}} F_y F_{ys} dy \\
	&= \int_{\mathbb{R}} F_y \left( \frac{y}{2}F_{yy} + \frac{1}{2}F_y + G_y \right) dy\\
	&= \frac{1}{4}\int_{\mathbb{R}} F_y^2 dy + \int_{\mathbb{R}} F_y G_y dy
\end{align*}
and
\begin{align*}
	\frac{d}{ds}\left[ \frac{1}{2}\frac{e^{-s}}{b(t(s))^2}\int_{\mathbb{R}}G^2 dy \right]
	&= -\frac{1}{b(t(s))^2} \frac{db}{dt}(t(s)) \int_{\mathbb{R}}G^2 dy
		-\frac{1}{2}\frac{e^{-s}}{b(t(s))^2}\int_{\mathbb{R}}G^2 dy
		+\frac{e^{-s}}{b(t(s))^2}\int_{\mathbb{R}}GG_s dy \\
	&= -\frac{1}{b(t(s))^2} \frac{db}{dt}(t(s)) \int_{\mathbb{R}}G^2 dy
		-\frac{1}{2}\frac{e^{-s}}{b(t(s))^2}\int_{\mathbb{R}}G^2 dy \\
	&\quad + \frac{e^{-s}}{b(t(s))^2}\int_{\mathbb{R}} G\left( \frac{y}{2}G_y + G \right) dy
		-\int_{\mathbb{R}}G^2 dy + \int_{\mathbb{R}}GF_{yy}dy + \int_{\mathbb{R}}GH dy \\ 
	&= -\frac{1}{b(t(s))^2} \frac{db}{dt}(t(s)) \int_{\mathbb{R}}G^2 dy
		+\frac{1}{4}\frac{e^{-s}}{b(t(s))^2}\int_{\mathbb{R}}G^2 dy \\
	&\quad - \int_{\mathbb{R}}G^2 dy - \int_{\mathbb{R}}F_yG_ydy + \int_{\mathbb{R}}GHdy. 
\end{align*}
Adding up the above two identities, one has
\begin{align}
\nonumber
	&\frac{d}{ds}\left[ \frac{1}{2} \int_{\mathbb{R}}
		\left( F_y^2 + \frac{e^{-s}}{b(t(s))^2}G^2 \right)dy \right]\\
\label{en2}
	&\quad = \frac{1}{4}\int_{\mathbb{R}}F_y^2dy
		+ \frac{1}{4}\frac{e^{-s}}{b(t(s))^2}\int_{\mathbb{R}}G^2 dy
		- \frac{1}{b(t(s))^2} \frac{db}{dt}(t(s)) \int_{\mathbb{R}}G^2 dy
		- \int_{\mathbb{R}}G^2 dy
		+ \int_{\mathbb{R}}GH dy. 
\end{align}
From \eqref{en1} and \eqref{en2}, we conclude that
\begin{align*}
	\frac{d}{ds}E_0(s)
	+\frac{1}{2}E_0(s) + \int_{\mathbb{R}}\left( \frac{1}{2}F_y^2 + G^2 \right)dy
	= R_0(s).
\end{align*}
This completes the proof.
\end{proof}


\subsection{Energy estimates for $n\ge 2$}
Next, we consider higher dimensional cases $n\ge 2$.
In this case, we cannot use the primitives \eqref{1_FG}.
Therefore, instead of \eqref{1_FG}, we define
\begin{align*}
	\hat{F}(s,\xi) = |\xi|^{-n/2-\delta}\hat{f}(s,\xi),\quad
	\hat{G}(s,\xi) = |\xi|^{-n/2-\delta}\hat{g}(s,\xi),\quad
	\hat{H}(s,\xi) = |\xi|^{-n/2-\delta}\hat{h}(s,\xi),
\end{align*}
where
$0<\delta<1$,
and
$\hat{f}(s,\xi)$ denotes the Fourier transform of $f(s,y)$ with respect to
the space variable.
First, to ensure that $\hat{F}, \hat{G}$ and $\hat{H}$ make sense as $L^2$-functions,
instead of Lemma \ref{lem_hardy}, we prove the following lemma.
\begin{lemma}\label{lem_hardy2}
Let $m>n/2+1$ and $f(y) \in H^{0,m}(\mathbb{R}^n)$ be a function satisfying
$\hat{f}(0) = (2\pi)^{-n/2} \int_{\mathbb{R}^n}f(y)dy = 0$.
Let
$\hat{F}(\xi) = |\xi|^{-n/2-\delta}\hat{f}(\xi)$
with some $0<\delta<1$.
Then,
there exists a constant $C(n,m,\delta)>0$ such that
\begin{align}
\label{hardy2}
	\| F \|_{L^2} \le C(n,m,\delta) \| f \|_{H^{0,m}}
\end{align}
holds.
\end{lemma}
\begin{proof}
By the Plancherel theorem, it suffices to show that
$\| \hat{F} \|_{L^2} \le C \| f \|_{H^{0,m}}$.
Using the definition of $\hat{F}$ and the condition $\hat{f}(0)=0$ , we compute
\begin{align*}
	\int_{\mathbb{R}^n}| \hat{F}(\xi) |^2 d\xi
	&= \int_{\mathbb{R}^n} |\xi|^{-n-2\delta} |\hat{f}(\xi)|^2 d\xi\\
	&= \int_{|\xi|\le 1} |\xi|^{-n-2\delta} |\hat{f}(\xi)|^2 d\xi
		+ \int_{|\xi|>1} |\xi|^{-n-2\delta} |\hat{f}(\xi)|^2 d\xi \\
	&= \int_{|\xi|\le 1} |\xi|^{-n-2\delta}
			\left| \int_0^1\frac{d}{d\theta}\hat{f}(\theta \xi) d\theta \right|^2 d\xi
		+ \int_{|\xi|>1} |\xi|^{-n-2\delta} |\hat{f}(\xi)|^2 d\xi \\
	&\le \| \nabla_{\xi}\hat{f} \|_{L^{\infty}}^2 \int_{|\xi|\le 1} |\xi|^{2-n-2\delta}d\xi
		+ \| \hat{f} \|_{L^2}^2\\
	&\le C(n,\delta) \left( \| \nabla_{\xi}\hat{f} \|_{L^{\infty}}^2 + \| \hat{f} \|_{L^2}^2 \right).
\end{align*}
Since $m>n/2+1$, we have
\begin{align*}
	\| \nabla_{\xi}\hat{f} \|_{L^{\infty}}
	= \| \widehat{yf} \|_{L^{\infty}}
	\le C \| y f \|_{L^1} \le C(n, m) \| (1+|y| )^{m} f\|_{L^2}
		\le C(n,m) \| f \|_{H^{0,m}}.
\end{align*}
Consequently, we obtain
\begin{align*}
	\| \hat{F} \|_{L^2}
	\le C(n,\delta) \left( \| \nabla_{\xi}\hat{f} \|_{L^{\infty}} + \| \hat{f} \|_{L^2} \right)
	\le C(n,m,\delta) \| f \|_{H^{0,m}},
\end{align*}
which completes the proof.
\end{proof}

We also notice that, 
for any small $\eta >0$, the inequality
\begin{align}
\nonumber
	\int_{\mathbb{R}^n} |\hat{f}|^2 d\xi
	& = \int_{|\xi| \ge \sqrt{\eta}^{-1}} |\hat{f}|^2 d\xi
		+ \int_{|\xi| < \sqrt{\eta}^{-1}} |\hat{f}|^2 d\xi \\
\nonumber
	&\le  \eta \int_{|\xi|\ge \sqrt{\eta}^{-1}} |\xi|^2 |\hat{f}|^2 d\xi
		+ \eta^{(2-n-2\delta)/2} \int_{|\xi| < \sqrt{\eta}^{-1}} |\xi|^{2-n-2\delta} |\hat{f}|^2 d\xi \\
\label{est_f}
	&\le \eta \int_{\mathbb{R}^n} |\xi|^2 |\hat{f}|^2 d\xi
		+  \eta^{(2-n-2\delta)/2} \int_{\mathbb{R}^n} |\xi|^2 |\hat{F}|^2 d\xi
\end{align}
holds.
This is proved by noting that
$2-n-2\delta <0$ 
(here we assumed that $n\ge 2$).
The above inequality enables us to control
$\| \hat{f} \|_{L^2}$
by
$\| |\xi|\hat{f} \|_{L^2}$ and $\| |\xi|\hat{F} \|_{L^2}$.
Moreover, the coefficient in front of $\| |\xi|\hat{f} \|_{L^2}$ can be taken
arbitrarily small.

By applying the Fourier transform to \eqref{eq_fg}, we obtain
\begin{align}
\label{eq_fg_F}
	\left\{\begin{array}{ll}
	\displaystyle
	\hat{f}_s + \frac{1}{2}\nabla_{\xi}\cdot \left( \xi \hat{f} \right) - \frac{n}{2} \hat{f} = \hat{g},
	&s>0, \xi \in\mathbb{R}^n,\\
	\displaystyle
	\frac{e^{-s}}{b(t(s))^2}
	\left( \hat{g}_s + \frac{1}{2}\nabla_{\xi}\cdot \left( \xi\hat{g} \right)
		-\left( \frac{n}{2}+1\right)\hat{g} \right) + \hat{g} = -|\xi|^2\hat{f} + \hat{h},
	&s>0, \xi \in\mathbb{R}^n.
	\end{array}\right.
\end{align}
By noting that
\[
	\frac{1}{2}\nabla_{\xi}\cdot \left( \xi \hat{f} \right)
		= \frac{\xi}{2}\cdot \nabla_{\xi}\hat{f} + \frac{n}{2}\hat{f},
\]
we rewrite \eqref{eq_fg_F} as
\begin{align*}
	\left\{ \begin{array}{ll}
	\displaystyle
	\hat{f}_s + \frac{\xi}{2}\cdot \nabla_{\xi} \hat{f} = \hat{g},
	&s>0, \xi \in\mathbb{R}^n,\\
	\displaystyle
	\frac{e^{-s}}{b(t(s))^2}
	\left( \hat{g}_s + \frac{\xi}{2} \cdot \nabla_{\xi}\hat{g}- \hat{g} \right)
	+ \hat{g} = -|\xi|^2\hat{f} + \hat{h},
	&s>0, \xi \in\mathbb{R}^n.
	\end{array}\right.
\end{align*}
Making use of this, we calculate
\begin{align*}
	\hat{F}_s &=
		|\xi|^{-n/2-\delta}\hat{f}_s\\
	&= |\xi|^{-n/2-\delta} \left( -\frac{\xi}{2}\cdot \nabla_{\xi}\hat{f} + \hat{g} \right)\\
	&= |\xi|^{-n/2-\delta} \left( -\frac{\xi}{2}\cdot \nabla_{\xi} \left( |\xi|^{n/2+\delta}\hat{F} \right)
		+ |\xi|^{n/2+\delta}\hat{G} \right) \\
	&= -\frac{\xi}{2}\cdot \nabla_{\xi}\hat{F} - \frac{1}{2}\left( \frac{n}{2}+\delta \right) \hat{F}
		+ \hat{G}
\end{align*}
and
\begin{align*}
	\frac{e^{-s}}{b(t(s))^2}\hat{G}_s&=
		\frac{e^{-s}}{b(t(s))^2}|\xi|^{-n/2-\delta}\hat{g}_s\\
	&= |\xi|^{-n/2-\delta}\left[
		\frac{e^{-s}}{b(t(s))^2}\left( -\frac{\xi}{2}\cdot\nabla_{\xi}\hat{g}+\hat{g}\right)
			-\hat{g}-|\xi|^2\hat{f}+\hat{h} \right]\\
	&= |\xi|^{-n/2-\delta}\left[
		\frac{e^{-s}}{b(t(s))^2}
		\left( -\frac{\xi}{2}\cdot \nabla_{\xi}\left( |\xi|^{n/2+\delta} \hat{G} \right)
				+|\xi|^{n/2+\delta}\hat{G} \right) \right.\\
	&\qquad\qquad\qquad
		\left. -|\xi|^{n/2+\delta}\hat{G} - |\xi|^{2+n/2+\delta}\hat{F} + |\xi|^{n/2+\delta}\hat{H}
		\right]\\
	&=\frac{e^{-s}}{b(t(s))^2}\left( -\frac{\xi}{2}\cdot \nabla_{\xi}\hat{G}
		-\frac{1}{2}\left( \frac{n}{2}+\delta-2 \right) \hat{G} \right) 
		-\hat{G} - |\xi|^2\hat{F}+\hat{H}.
\end{align*}
Hence, $\hat{F}$ and $\hat{G}$ satisfy the following system.
\begin{align*}
	\left\{ \begin{array}{ll}
	\displaystyle \hat{F}_s + \frac{\xi}{2}\cdot \nabla_{\xi}\hat{F}
		+\frac{1}{2}\left( \frac{n}{2} + \delta \right) \hat{F} = \hat{G},
	&s>0, \xi \in \mathbb{R}^n,\\
	\displaystyle \frac{e^{-s}}{b(t(s))^2}\left( \hat{G}_s + \frac{\xi}{2}\cdot \nabla_{\xi} \hat{G}
		+ \frac{1}{2} \left( \frac{n}{2}+\delta-2 \right) \hat{G} \right) + \hat{G}
		= -|\xi|^2 \hat{F} + \hat{H},
	&s>0, \xi\in\mathbb{R}^n.
	\end{array} \right.
\end{align*}

We consider the following energy.
\begin{align*}
	E_0(s) &= {\rm Re} \int_{\mathbb{R}^n}
		\left( \frac{1}{2}\left( |\xi|^2 |\hat{F}|^2 + \frac{e^{-s}}{b(t(s))^2} |\hat{G}|^2 \right)
			+ \frac{1}{2} |\hat{F}|^2 + \frac{e^{-s}}{b(t(s))^2}\hat{F} \bar{\hat{G}}  \right) d\xi,\\
	E_1(s) &= \int_{\mathbb{R}^n}
		\left( \frac{1}{2}\left( |\nabla_y f |^2 + \frac{e^{-s}}{b(t(s))^2}g^2 \right)
		+ \left( \frac{n}{4} + 1 \right)
			\left( \frac{1}{2} f^2 + \frac{e^{-s}}{b(t(s))^2}fg \right) \right) dy,\\
	E_2(s) &= \int_{\mathbb{R}^n}
		|y|^{2m} \left[
			\frac{1}{2}\left( |\nabla_y f|^2 + \frac{e^{-s}}{b(t(s))^2}g^2 \right)
				+\frac{1}{2}f^2 + \frac{e^{-s}}{b(t(s))^2}fg \right] dy.
\end{align*}
By using Lemma \ref{lem_b0b1} again,
the following equivalents are valid for $s\ge s_1$ with sufficiently large $s_1$.
\begin{align}
\nonumber
	E_0(s) & \sim \int_{\mathbb{R}^n}
		\left( |\xi|^2 |\hat{F}|^2 + \frac{e^{-s}}{b(t(s))^2} |\hat{G}|^2 + |\hat{F}|^2 \right) d\xi,\\
\label{equi2}
	E_1(s) & \sim \int_{\mathbb{R}^n}
		\left( | \nabla_y f |^2 + \frac{e^{-s}}{b(t(s))^2} g^2 + f^2 \right) dy,\\
\nonumber
	E_2(s) &\sim \int_{\mathbb{R}^n}
		|y|^{2m} \left[ | \nabla_y f |^2 + \frac{e^{-s}}{b(t(s))^2} g^2 + f^2 \right] dy.
\end{align}

Then, in a similar way to the case $n=1$, we obtain the following energy identities.
\begin{lemma}\label{lem_en20}
We have
\begin{align}
\label{e0n2}
	\frac{d}{ds}E_0(s) + \delta E_0(s) + L_0(s)
	= R_0(s),
\end{align}
where
\begin{align*}
	L_0(s) &= \frac{1}{2} \int_{\mathbb{R}^n} |\xi|^2 |\hat{F}|^2 d\xi
		+\int_{\mathbb{R}^n} |\hat{G}|^2 d\xi,\\
	R_0(s) &= \frac{3}{2} \frac{e^{-s}}{b(t(s))^2}\int_{\mathbb{R}^n} |\hat{G}|^2d\xi
	- \frac{1}{b(t(s))^2} \frac{db}{dt}(t(s)) 
		{\rm Re}\int_{\mathbb{R}^n} \left( 2\hat{F} +\hat{G} \right) \bar{\hat{G}} d\xi
		+ {\rm Re} \int_{\mathbb{R}^n} \left( \hat{F} + \hat{G} \right) \bar{\hat{H}} d\xi.
\end{align*}
Moreover, we have
\begin{align}
\label{e1n2}
	\frac{d}{ds}E_1(s) + \delta E_1(s) + L_1(s)
	=R_1(s),
\end{align}
where
\begin{align*}
	L_1(s) &= \frac{1}{2}(1-\delta) \int_{\mathbb{R}^n}|\nabla_y f|^2 dy
		+ \int_{\mathbb{R}^n} g^2 dy
		- \left( \frac{n}{4}+\frac{\delta}{2} \right)\left(\frac{n}{4}+1\right)
			\int_{\mathbb{R}^n}f^2 dy,\\
\nonumber
	R_1(s) &= \left( \frac{n}{2}+\delta \right)\left( \frac{n}{4}+1 \right)
			\frac{e^{-s}}{b(t(s))^2}\int_{\mathbb{R}^n} fg dy
		+ \frac{1}{2}(n+3+\delta) \frac{e^{-s}}{b(t(s))^2} \int_{\mathbb{R}^n} g^2 dy \\
	&\quad
		- \frac{1}{b(t(s))^2} \frac{db}{dt}(t(s))
			\int_{\mathbb{R}^n} \left( 2 \left(\frac{n}{4}+1\right) f + g \right) g dy
		+ \int_{\mathbb{R}^n} \left( \left( \frac{n}{4} + 1 \right) f + g \right) h dy.
\end{align*}
Furthermore, we have
\begin{align}
\label{e2n2}
	&\frac{d}{ds}E_2(s)+ (\tilde{\delta} - \eta) E_2(s) + L_2(s) 
	=R_2(s),
\end{align}
where
$\tilde{\delta} = m - n/2$,
$\eta \in (0,\tilde{\delta})$ is an arbitrary number,
\begin{align*}
	L_2(s) &= \frac{\eta}{2}\int_{\mathbb{R}^n}|y|^{2m} f^2 dy
	 + \frac{1}{2} ( \eta + 1) \int_{\mathbb{R}^n} |y|^{2m} |\nabla_y f|^2 dy
	 + \int_{\mathbb{R}^n} |y|^{2m} g^2 dy\\
	 &\quad + 2m \int_{\mathbb{R}^n}|y|^{2m-2} (y\cdot \nabla_y f)(f+g) dy,\\
	R_2(s) &= -\eta \frac{e^{-s}}{b(t(s))^2} \int_{\mathbb{R}^n} |y|^{2m} f g dy
		- \frac{1}{2}(\eta - 3) \frac{e^{-s}}{b(t(s))^2} \int_{\mathbb{R}^n} |y|^{2m} g^2 dy\\
	&\quad - \frac{1}{b(t(s))^2} \frac{db}{dt}(t(s))
		\int_{\mathbb{R}^n}|y|^{2m} (2f +g) g dy
		+ \int_{\mathbb{R}^n} |y|^{2m} (f+g) h dy.
\end{align*}
\end{lemma}
\begin{proof}
The proofs of \eqref{e1n2} and \eqref{e2n2} are the almost same as
that of \eqref{e0n2},
and we only prove \eqref{e0n2}.
First, we calculate
\begin{align*}
	\frac{d}{ds} \left[ \frac{1}{2} \int_{\mathbb{R}^n} |\hat{F}|^2\, d\xi \right]
	&= {\rm Re\,} \int_{\mathbb{R}^n}
		\left( -\frac{\xi}{2} \cdot \nabla_{\xi} \hat{F}
			- \frac{1}{2}\left( \frac{n}{2}+\delta \right) \hat{F} + \hat{G} \right)
			\bar{\hat{F}} \, d\xi \\
	&= -\frac{\delta}{2} \int_{\mathbb{R}^n} |\hat{F}|^2\, d\xi
		+ {\rm Re\,} \int_{\mathbb{R}^n} \bar{\hat{F}} \hat{G}\, d\xi
\end{align*}
and
\begin{align*}
	\frac{d}{ds}\left[ \frac{e^{-s}}{b(t(s))^2} {\rm Re\,}
		\int_{\mathbb{R}^n} \hat{F} \bar{\hat{G}} \, d\xi \right]
	&= - \frac{2}{b(t(s))^2} \frac{db}{dt}(t(s))
		{\rm Re\,} \int_{\mathbb{R}^n} \hat{F} \bar{\hat{G}} \,d\xi
		- \frac{e^{-s}}{b(t(s))^2} 
		{\rm Re\,} \int_{\mathbb{R}^n} \hat{F} \bar{\hat{G}} \,d\xi \\
	&\quad + \frac{e^{-s}}{b(t(s))^2}
		{\rm Re\,} \int_{\mathbb{R}^n}
			\left( \hat{F}_s \bar{\hat{G}} + \bar{\hat{F}} \hat{G}_s \right) \,d\xi \\
	&=
		- \frac{2}{b(t(s))^2} \frac{db}{dt}(t(s))
		{\rm Re\,} \int_{\mathbb{R}^n} \hat{F} \bar{\hat{G}} \,d\xi
		 - \delta \frac{e^{-s}}{b(t(s))^2}
		 {\rm Re\,} \int_{\mathbb{R}^n} \hat{F} \bar{\hat{G}} \,d\xi \\
	&\quad
		+ \frac{e^{-s}}{b(t(s))^2} \int_{\mathbb{R}^n} |\hat{G}|^2 \,d\xi 
		-  {\rm Re\,} \int_{\mathbb{R}^n} \hat{F} \bar{\hat{G}} \,d\xi \\
	&\quad - \int_{\mathbb{R}^n} |\xi|^2 |\hat{F}|^2 \, d\xi
		+ {\rm Re\,} \int_{\mathbb{R}^n} \hat{F} \bar{\hat{H}}\, d\xi.
\end{align*}
Adding up these identities, we see that
\begin{align}
\nonumber
	&\frac{d}{ds}\left[ {\rm Re\,}
		\int_{\mathbb{R}^n}
		\left( |\hat{F}|^2 + \frac{e^{-s}}{b(t(s))^2}  \hat{F} \bar{\hat{G}} \right) \, d\xi \right] \\
\nonumber
	&\quad = -\frac{\delta}{2} \int_{\mathbb{R}^n} |\hat{F}|^2\,d\xi
		- \frac{2}{b(t(s))^2} \frac{db}{dt}(t(s)) {\rm Re\,}
			\int_{\mathbb{R}^n} \hat{F} \bar{\hat{G}}\, d\xi
		- \delta \frac{e^{-s}}{b(t(s))^2}
			{\rm Re\,} \int_{\mathbb{R}^n} \hat{F} \bar{\hat{G}}\, d\xi \\
\label{en21}
	&\qquad + \frac{e^{-s}}{b(t(s))^2} \int_{\mathbb{R}^n} | \hat{G} |^2\, d\xi
		- \int_{\mathbb{R}^n} |\xi|^2 |\hat{F}|^2 \,d\xi
		+ {\rm Re\,} \int_{\mathbb{R}^n} \hat{F} \bar{\hat{H}}\, d\xi
\end{align}
We also have
\begin{align*}
	\frac{d}{ds} \left[ \frac{1}{2} \int_{\mathbb{R}^n} |\xi|^2 |\hat{F}|^2\, d\xi \right]
	&= {\rm Re\,} \int_{\mathbb{R}^n} |\xi|^2
		\left( -\frac{\xi}{2} \cdot \nabla_{\xi} \hat{F}
			- \frac{1}{2}\left( \frac{n}{2}+\delta \right) \hat{F} + \hat{G} \right)
			\bar{\hat{F}} \, d\xi \\
	&= \frac{1}{2}(1-\delta) \int_{\mathbb{R}^n} |\xi|^2 |\hat{F}|^2\,d\xi
		+ {\rm Re\,} \int_{\mathbb{R}^n} |\xi|^2 \bar{\hat{F}} \hat{G}\, d\xi
\end{align*}
and
\begin{align*}
	\frac{d}{ds}\left[ \frac{1}{2} \frac{e^{-s}}{b(t(s))^2}
		\int_{\mathbb{R}^n} |\hat{G}|^2\, d\xi \right]
	&= -\frac{1}{b(t(s))^2} \frac{db}{dt}(t(s)) \int_{\mathbb{R}^n} |\hat{G}|^2\,d\xi
		-\frac{1}{2} \frac{e^{-s}}{b(t(s))^2} \int_{\mathbb{R}^n} |\hat{G}|^2\,d\xi \\
	&\quad	
		+ \frac{e^{-s}}{b(t(s))^2} {\rm Re\,} \int_{\mathbb{R}^n} \hat{G}_s \bar{\hat{G}}\,d\xi \\
	&= -\frac{1}{b(t(s))^2} \frac{db}{dt}(t(s)) \int_{\mathbb{R}^n} |\hat{G}|^2\,d\xi
		+ \frac{1}{2}(1-\delta) \frac{e^{-s}}{b(t(s))^2} \int_{\mathbb{R}^n} |\hat{G}|^2\,d\xi \\
	&\quad - \int_{\mathbb{R}^n} |\hat{G}|^2\,d\xi
		- {\rm Re\,} \int_{\mathbb{R}^n} |\xi|^2 \hat{F} \bar{\hat{G}} \,d\xi
		+ {\rm Re\,} \int_{\mathbb{R}^n} \hat{G} \bar{\hat{H}} \,d\xi.
\end{align*}
Summing up the above identities, we have
\begin{align}
\nonumber
	&\frac{d}{ds}\left[ \frac{1}{2} \frac{e^{-s}}{b(t(s))^2}
		\int_{\mathbb{R}^n} |\hat{G}|^2\, d\xi \right] \\
\nonumber
	&\quad =
		\frac{1}{2}(1-\delta) \int_{\mathbb{R}^n} |\xi|^2 |\hat{F}|^2\,d\xi
		- \frac{1}{b(t(s))^2} \frac{db}{dt}(t(s)) \int_{\mathbb{R}^n} |\hat{G}|^2\,d\xi \\
\label{en22}
	&\qquad
		+ \frac{1}{2}(1-\delta)  \frac{e^{-s}}{b(t(s))^2} \int_{\mathbb{R}^n} |\hat{G}|^2\,d\xi 
		- \int_{\mathbb{R}^n} |\hat{G}|^2\,d\xi
		+ {\rm Re\,} \int_{\mathbb{R}^n} \hat{G} \bar{\hat{H}} \,d\xi.
\end{align}
From \eqref{en21} and \eqref{en22}, we conclude \eqref{e0n2}.
\end{proof}


\subsection{Proof of Proposition \ref{prop_ap}}
In either case when $n=1$ or $n \ge 2$, we have proved energy identities
of $E_j(s)$ with remainder terms $R_j\ (j=0,1,2)$.
Hereafter, we unify the both cases
and complete the proof of Proposition \ref{prop_ap}.
We define
\begin{align*}
	E_3(s) = \frac{1}{2}\frac{e^{-s}}{b(t(s))^2}\left( \frac{d\alpha}{ds}(s) \right)^2
		+ e^{-2\lambda s}\alpha(s)^2
\end{align*}
and
\begin{align*}
	E_4(s) = C_0 E_0(s) + C_1 E_1(s) + E_2(s) + E_3(s),
\end{align*}
where
$\lambda > 0$ is determined later,
and $C_0, C_1$ are positive constants such that
$1 \ll C_1 \ll C_0$.
By recalling the equivalences \eqref{equi1} and \eqref{equi2}, 
the following equivalence is valid for $s\ge s_1$:
\begin{align}
\label{equi4}
	E_4(s) \sim
	\| f(s) \|_{H^{1,m}}^2 + \frac{e^{-s}}{b(t(s))^2}\| g(s) \|_{H^{0,m}}^2
					+ \frac{e^{-s}}{b(t(s))^2}\left( \frac{d\alpha}{ds}(s) \right)^2
					+ e^{-2\lambda s}\alpha(s)^2.
\end{align}
To obtain the energy estimate of $E_4(s)$, we first notice the following lemma.
\begin{lemma}\label{lem_alpha2} 
We have
\begin{align*}
	&\frac{d}{ds} E_3(s) + 2 \lambda E_3(s) +\left( \frac{d\alpha}{ds}(s) \right)^2 = R_3(s),
\end{align*}
where
\begin{align}
\nonumber
	R_3(s) &=  \frac{1}{2}(2\lambda +1 ) \frac{e^{-s}}{b(t(s))^2}\left( \frac{d\alpha}{ds}(s) \right)^2
			- \frac{1}{b(t(s))^2} \frac{db}{dt}(t(s)) \left( \frac{d\alpha}{ds}(s) \right)^2 \\
\label{rp}
	&\quad + \frac{d\alpha}{ds}(s) \left( \int_{\mathbb{R}^n} r(s,y) dy \right)
		+ 2 e^{-2 \lambda s} \alpha(s) \frac{d\alpha}{ds}(s).
\end{align}
\end{lemma}

Then, we can also see the following energy estimate.
\begin{lemma}\label{lem_en3}
We have
\begin{align}
\label{en3_est}
	&\frac{d}{ds}E_4(s) + 2 \lambda E_4(s)
	+ L_4(s)
	= R_4(s),
\end{align}
where
\begin{align*}
	L_4(s) &= \left( \frac{1}{2} - 2 \lambda \right) \left( C_0E_0(s)+C_1E_1(s)+E_2(s) \right)
			+ C_0L_0(s) + C_1 L_1(s) + L_2(s) + \left( \frac{d\alpha}{ds}(s) \right)^2,
\end{align*}
for $n=1$,
\begin{align*}
	L_4(s) &= C_0 (\delta-2\lambda) E_0(s)
		+ C_1 (\delta-2\lambda) E_1(s) + (\tilde{\delta}-\eta-2\lambda) E_2(s) \\
		&\quad + C_0L_0(s) + C_1 L_1(s) + L_2(s) + \left( \frac{d\alpha}{ds}(s) \right)^2
\end{align*}
for $n\ge 2$,
and
\begin{align*}
	R_4(s) &= C_0R_0(s)+C_1R_1(s) +R_2(s) + R_3(s).
\end{align*}
Here $R_0,R_1,R_2$ and $L_0,L_1,L_2$ are defined in
Lemmas \ref{lem_en0}\ $(n=1)$ and \ref{lem_en20}\ $(n\ge 2)$,
and
$R_3$ is defined by \eqref{rp}.
\end{lemma}
Then, by the Schwarz inequality and the inequality \eqref{est_f},
we obtain the following lower estimate of $L_4$.
Here we recall that
$\delta \in (0,1)$ is an arbitrary number,
$\tilde{\delta}=m-n/2$ and
$\eta>0$ is an arbitrary small number.
\begin{lemma}\label{lem_l4}
If
$0<\lambda \le 1/4\ (n=1)$,
$0<\lambda < \min \{ \frac{1}{2}, \frac{m}{2}-\frac{n}{4} \} \ (n\ge 2)$,
then
\begin{align*}
	L_4(s) & \ge C \left( \| f(s) \|_{H^{1,m}}^2 + \| g(s) \|_{H^{0,m}}^2 + \left( \frac{d\alpha}{ds}(s) \right)^2 \right)
\end{align*}
holds for $s \ge s_1$. 
\end{lemma}
\begin{proof}
Let $\lambda$ satisfy
$0<\lambda \le 1/4\ (n=1)$ and
$0<\lambda < \min \{ \frac{1}{2}, \frac{m}{2}-\frac{n}{4} \} \ (n\ge 2)$.
We take the parameter
$\delta$ so that
$2\lambda < \delta < 1$.
Then, recalling that $\tilde{\delta} = m - n/2$, we have
$2 \lambda < \min\{ \delta, \tilde{\delta} \}$.
We also note that
the equivalences \eqref{equi1} and \eqref{equi2} of $E_0(s), E_1(s), E_2(s)$
yield the positivity of the first three terms of $L_4(s)$ for $s \ge s_1$. 
Therefore, it suffices to consider the terms
$L_0(s), L_1(s)$ and $L_2(s)$.
When $n = 1$, noting $F_y = f$ and applying the Schwarz inequality, we easily have
\[
	\int_{\mathbb{R}} f^2 dy = \int_{\mathbb{R}} F_y^2 dy
\]
and
\[
	2\left| \int_{\mathbb{R}}y f_y (f+g) dy \right| 
	\le \frac{1}{4} \int_{\mathbb{R}} y^2 f_y^2 dy + 8 \int_{\mathbb{R}} (f^2 + g^2) dy.
\]
Hence, taking $C_1>8$ and $C_0>2C_1$,
we obtain the desired estimate.

Next, when $n\ge 2$, we note that, for any small $\mu >0$, we have
\begin{align*}
	& \left| \int_{\mathbb{R}^n} |y|^{2m-2} (y\cdot \nabla_y f)(f+g) dy \right| \\
	&\quad \le \mu \int_{\mathbb{R}^n} |y|^{2m} |\nabla_yf|^2 dy
			+ 8\mu^{-1} \int_{\mathbb{R}^n} |y|^{2m-2} ( f^2+g^2 ) dy
\end{align*}
and
\begin{align*}
	\mu^{-1} \int_{\mathbb{R}^n} |y|^{2m-2} ( f^2+g^2 ) dy
	& = \mu^{-1} \int_{|y|>\mu^{-1}} |y|^{2m-2} (f^2+g^2) dy
		+ \mu^{-1} \int_{|y|\le \mu^{-1}} |y|^{2m-2} (f^2+g^2) dy \\
	& \le \mu \int_{|y| > \mu^{-1}} |y|^{2m} (f^2+g^2) dy
		+ \mu^{-2m+1} \int_{ |y| \le \mu^{-1}} (f^2 + g^2) dy\\
	& \le \mu \int_{\mathbb{R}^n} |y|^{2m} (f^2+g^2) dy
		+ \mu^{-2m+1} \int_{\mathbb{R}^n} (f^2 + g^2) dy.
\end{align*}
We take $\mu$ sufficiently small so that $\mu \ll \eta$ and then
$C_0, C_1$ sufficiently large so that $\mu^{-2m+1} \ll C_1 \ll C_0$.
Then, applying \eqref{est_f} to estimate the last term,
we have the desired estimate.
\end{proof}

Finally, we put
\begin{align*}
	E_5(s) = E_4(s) + \frac{1}{2}\alpha(s)^2
		+ \frac{e^{-s}}{b(t(s))^2}\alpha(s) \frac{d\alpha}{ds}(s).
\end{align*}
Then, we easily obtain
\begin{lemma}\label{lem_compatible}
There exists $s_2 \ge s_1$ such that we have 
\begin{align*}
	&E_5(s) \sim
	\| f(s) \|_{H^{1,m}}^2 + \frac{e^{-s}}{b(t(s))^2} \| g(s) \|_{H^{0,m}}^2
		+ \alpha(s)^2 +  \frac{e^{-s}}{b(t(s))^2} \left( \frac{d\alpha}{ds}(s) \right)^2,\\
	&E_5(s) + \| g(s) \|_{H^{0,m}}^2 + \left( \frac{d\alpha}{ds}(s) \right)^2
		\sim \| f(s) \|_{H^{1,m}}^2 +\| g(s) \|_{H^{0,m}}^2 + \alpha(s)^2
			+ \left( \frac{d\alpha}{ds}(s) \right)^2
\end{align*}
for $s\ge s_2$. 
\end{lemma}
\begin{proof} 
By the Schwarz inequality, we have
\[
	\left| \frac{e^{-s}}{b(t(s))^2}\alpha(s) \frac{d\alpha}{ds}(s) \right|
	\le C(\tilde{\eta}) \frac{e^{-s}}{b(t(s))^2} \alpha(s)^2
		+ \tilde{\eta} \frac{e^{-s}}{b(t(s))^2} \left( \frac{d\alpha}{ds}(s) \right)^2,
\]
where
$\tilde{\eta} > 0$ is a small number determined later.
By the equivalence \eqref{equi4} of $E_4(s)$ and
taking $\tilde{\eta}$ sufficiently small,
we control the second term of the right-hand side and have
\[
	\tilde{\eta} \frac{e^{-s}}{b(t(s))^2} \left( \frac{d\alpha}{ds}(s) \right)^2
	\le \frac{1}{2} E_4(s)
\]
for $s \ge s_1$.
On the other hand, by Lemma \ref{lem_b0b1} and taking $s_2 \ge s_1$ sufficiently large,
we estimate the first term as
\[
	C(\tilde{\eta}) \frac{e^{-s}}{b(t(s))^2} \alpha(s)^2
	\le \frac{1}{4} \alpha(s)^2
\]
for $s \ge s_2$.
Combining them, we conclude that
\[
	E_5(s) \ge \frac12 E_4(s) + \frac14 \alpha(s)^2
\]
holds for $s \ge s_2$.
Then, using the lower bound \eqref{equi4} again,
we have the lower bound
\[
	\| f(s) \|_{H^{1,m}}^2 + \frac{e^{-s}}{b(t(s))^2} \| g(s) \|_{H^{0,m}}^2
		+ \alpha(s)^2 +  \frac{e^{-s}}{b(t(s))^2} \left( \frac{d\alpha}{ds}(s) \right)^2
	\le C E_5(s)
\]
for $s\ge s_2$.
The upper bound of $E_5(s)$ immediately follows from
the equivalence \eqref{equi4} of $E_4(s)$ and we have the first assertion.
The second assertion is also directly proved from the first one.
\end{proof}

By using \eqref{alpha_ddt}, we also have
\begin{align*}
	\frac{d}{ds}
		\left[ \frac{1}{2}\alpha(s)^2
			+\frac{e^{-s}}{b(t(s))^2}\alpha(s) \frac{d\alpha}{ds}(s) \right]
	&= \frac{e^{-s}}{b(t(s))^2}\left( \frac{d\alpha}{ds}(s) \right)^2
		- \frac{2}{b(t(s))^2}\alpha(s) \frac{db}{dt}(t(s)) \frac{d\alpha}{ds}(s) \\
	&\quad +\alpha(s) \left( \int_{\mathbb{R}^n} r(s,y)dy \right) \\
	&=: \tilde{R}_5(s). 
\end{align*}
Letting
$R_5(s) = R_4(s) + \tilde{R}_5(s)$,
we obtain
\begin{align}
\label{e5}
	\frac{d}{ds}E_5(s) + 2\lambda E_4(s) + L_4(s) = R_5(s).
\end{align}
We give an estimate for the remainder term $R_5(s)$:
\begin{lemma}[Estimate for the remainder terms]\label{lem_rema}
Let $\lambda_0, \lambda_1$ be
\begin{align}
\label{lambda_0}
	\lambda_0 = \min \left\{ \frac{1-\beta}{1+\beta},
						\frac{\gamma}{1+\beta}- \frac{1}{2}, \frac{\nu}{1+\beta}-1 \right\}
\end{align}
(where we interpret $1/(1+\beta)$ as an arbitrary large number when $\beta=-1$)
and
\begin{align}
\label{lambda_1}
	\lambda_1 = \begin{cases}
		\displaystyle \frac{1}{2} \min_{i=1,\ldots,k}
			\left\{ p_{i1} +2p_{i2} + \left( 3 - \frac{2\beta}{1+\beta} \right)p_{i3} - 3 \right\},
			& n=1,\\[7pt]
		\displaystyle \frac{n}{2} \left( p - 1- \frac{2}{n} \right), & n\ge 2
	\end{cases}
\end{align}
(where we interpret
$- 2\beta p_{i3}/(1+\beta)$
as an arbitrary large number when $p_{i3} \neq 0$ and $\beta = -1$).
Then, there exists $s_0 \ge s_2$ such that we have the following estimates: 
\begin{itemize}
\item[{\rm (i)}]
When $n =1$, $R_4(s)$ and $R_5(s)$ satisfy
\begin{align*}
	|R_4(s)| &\le \tilde{\eta} L_4(s)
		+ C(\tilde{\eta}) e^{-2\lambda_0 s} E_5(s) 
		+ C(\tilde{\eta}) e^{-2\lambda_1 s}
			\sum_{i=1}^k E_5(s)^{p_{i1}+p_{i2}}
				\left( E_5(s)^{p_{i3}} + L_4(s)^{p_{i3}} \right),\\
	|R_5(s)| &\le \tilde{\eta} L_4(s)
		+ C(\tilde{\eta}) e^{-\lambda_0 s} E_5(s)
		+ C(\tilde{\eta}) e^{-2\lambda_1 s}
			\sum_{i=1}^k E_5(s)^{p_{i1}+p_{i2}}
			\left( E_5(s)^{p_{i3}} + L_4(s)^{p_{i3}} \right) \\
		&\quad + C(\tilde{\eta}) e^{-\lambda_1 s}
			\sum_{i=1}^k
			E_5(s)^{(p_{i1}+p_{i2}+p_{i3}+1)/2}
\end{align*}
for $s \ge s_0$,
where $\tilde{\eta}>0$ is an arbitrary small number.
\item[{\rm (ii)}]
When $n \ge 2$, $R_4(s)$ and $R_5(s)$ satisfy
\begin{align*}
	|R_4(s)| &\le \tilde{\eta} L_4(s)
		+ C(\tilde{\eta}) e^{-2\lambda_0 s} E_5(s)
		+ C(\tilde{\eta}) e^{-2\lambda_1 s} E_5(s)^p,\\
	|R_5(s)| &\le \tilde{\eta} L_4(s)
		+ C(\tilde{\eta}) e^{-\lambda_0 s} E_5(s)
		+ C(\tilde{\eta}) e^{-2\lambda_1 s} E_5(s)^p
		+ C(\tilde{\eta}) e^{-\lambda_1 s} E_5(s)^{(p+1)/2}
\end{align*}
for $s \ge s_0$,
where $\tilde{\eta}>0$ is an arbitrary small number. 
\end{itemize}
\end{lemma}
We postpone the proof of this lemma until the next section,
and now we completes the
proofs of Proposition \ref{prop_ap} and Theorem \ref{thm}.
We first consider the case $n \ge 2$.
Taking $\tilde{\eta} = 1/2$ in Lemma \ref{lem_rema} and
using \eqref{e5} and Lemmas \ref{lem_l4}, \ref{lem_compatible}, we have
\begin{align}
\label{este52}
	\frac{d}{ds}E_5(s)
	\le C e^{-\lambda_0 s} E_5(s)
		+ C e^{-2\lambda_1 s} E_5(s)^p
		+ C e^{-\lambda_1 s} E_5(s)^{(p+1)/2}
\end{align}
for $s \ge s_0$.
Let
\[
	\Lambda (s) := \exp \left( -C\int_{s_0}^s e^{-\lambda_0 \tau}\, d\tau \right).
\]
We note that
$e^{-Ce^{-\lambda_0 s_0}/\lambda_0} \le \Lambda (s) \le 1$
for $s \ge s_0$ and $\Lambda(s_0) = 1$.
Multiplying \eqref{este52} by $\Lambda(s)$
and integrating it over $[s_0, s]$, we see that
\[
	\Lambda(s) E_5(s) \le E_5(s_0)
		+ C \int_{s_0}^s
		\left[ \Lambda(\tau) e^{-2\lambda_1 \tau} E_5(\tau)^p
			+ \Lambda(\tau) e^{-\lambda_1 \tau} E_5(\tau)^{(p+1)/2}
						\right] \, d\tau
\]
holds for $s \ge s_0$.
Putting
\[
	M(s) := \sup_{s_0 \le \tau \le s} E_5(\tau),
\]
we further obtain
\begin{align}
\label{estm}
	M(s) \le C M(s_0)
		+ C(s_0, \lambda_0, \lambda_1) \left( M(s)^p + M(s)^{(p+1)/2} \right)
\end{align}
for $s \ge s_0$.
On the other hand, we easily estimate $M(s_0)$ as
\begin{align}
\nonumber
	M(s_0) &\le C(s_0) \left( \| (f(s_0), g(s_0)) \|_{H^{1,m}\times H^{0,m}}^2
				+ \alpha(s_0)^2 + \dot{\alpha}(s_0)^2 \right) \\
\nonumber
		& \le C(s_0) \| (v(s_0), w(s_0)) \|_{H^{1,m}\times H^{0,m}}^2 \\
\label{este5}
	&\le C(s_0) \varepsilon^2 \| (v_0, w_0) \|_{H^{1,m}\times H^{0,m}}^2
\end{align}
by using the local existence result
(see the proof of Proposition \ref{prop_loc}).
Combining \eqref{estm} with \eqref{este5}, we have
\[
	M(s) \le
		C_2 \varepsilon^2 \| (v_0, w_0) \|_{H^{1,m}\times H^{0,m}}^2 
		+ C_2 \left( M(s)^p + M(s)^{(p+1)/2} \right)
\]
for $s \ge s_0$ with some constant $C_2 >0$.
Let $\varepsilon_1$ be
\[
	\varepsilon_1 := \left( \sqrt{C_2} 2^{(p+1)/4} \right)^{-1}.
\]
Then,
a direct calculation implies
\begin{align*}
	2 C_2 \varepsilon^2 I_0
		> C_2 \varepsilon^2 I_0
		+ C_2 \left[
			(2C_2 \varepsilon^2 I_0 )^p
			+(2C_2 \varepsilon^2 I_0 )^{(p+1)/2}
		\right]
\end{align*}
holds for $\varepsilon \in (0,\varepsilon_1]$,
where $I_0=\| (v_0, w_0) \|_{H^{1,m}\times H^{0,m}}^2$.
Combining this with
$M(s_0) \le C_2 \varepsilon^2 \| (v_0, w_0) \|_{H^{1,m}\times H^{0,m}}^2$
and the continuity of
$M(s)$ with respect to $s$, we conclude that
\begin{align}
\label{est_m}
	M(s) \le 2 C_2\varepsilon^2 \| (v_0, w_0) \|_{H^{1,m}\times H^{0,m}}^2
\end{align}
holds for
$s \ge s_0$ and $\varepsilon \in (0,\varepsilon_1]$.
Therefore, from Lemma \ref{lem_compatible}, we obtain
\[
	\| f(s) \|_{H^{1,m}}^2 + \frac{e^{-s}}{b(t(s))^2} \| g(s) \|_{H^{0,m}}^2
		+ \alpha(s)^2 +  \frac{e^{-s}}{b(t(s))^2} \left( \frac{d\alpha}{ds}(s) \right)^2
	\le C \varepsilon^2 \| (v_0, w_0) \|_{H^{1,m}\times H^{0,m}}^2
\]
for $s \ge s_0$ and $\varepsilon \in (0,\varepsilon_1]$.
This implies
\begin{align}
\label{aprivw}
	\| v(s) \|_{H^{1,m}}^2 + \frac{e^{-s}}{b(t(s))^2} \| w(s) \|_{H^{0,m}}^2
	\le C_{\ast} \varepsilon^2 \| (v_0, w_0) \|_{H^{1,m}\times H^{0,m}}^2
\end{align}
with some constant
$C_{\ast}>0$.
This completes the proof of Proposition \ref{prop_ap}.

When $n=1$, to control the additional term
$E_5(s)^{p_{i1}}L_4(s)$ appearing in the estimate of $R_5(s)$,
we use \eqref{e5} as
\begin{align*}
	\frac{d}{ds}E_5(s)
	+ L_4(s)
	&\le C e^{-\lambda_0 s} E_5(s)
		+ C(\tilde{\eta}) e^{-2\lambda_1 s}
			\sum_{i=1}^k E_5(s)^{p_{i1}+p_{i2}}
			\left( E_5(s)^{p_{i3}} + L_4(s)^{p_{i3}} \right) \\
		&\quad + C(\tilde{\eta}) e^{-\lambda_1 s}
			\sum_{i=1}^k
			E_5(s)^{(p_{i1}+p_{i2}+p_{i3}+1)/2}
\end{align*}
instead of \eqref{este52}.
In the same way as before, we multiply the both sides
by $\Lambda(s)$ and integrate it over $[s_0, s]$
to obtain
\begin{align*}
	&\Lambda(s) E_5(s)
		+ \int_{s_0}^s \Lambda (\tau) L_4(\tau)\,d\tau \\
	& \quad \le E_5(s_0)
		+ C \sum_{\substack{i=1,\ldots, k\\ p_{i3}=1}} \int_{s_0}^s
			\Lambda(\tau) e^{-2\lambda_1 \tau}
				E_5(\tau)^{p_{i1}+p_{i2}}L_4(\tau)^{p_{i3}}d\tau \\
	&\qquad + C \sum_{i=1}^k \int_{s_0}^s
		\left[ \Lambda(\tau) e^{-2\lambda_1 \tau} E_5(\tau)^{p_{i1}+p_{i2}+p_{i3}}
			+ \Lambda(\tau) e^{-\lambda_1 \tau}
				E_5(\tau)^{(p_{i1}+p_{i2}+p_{i3}+1)/2}
						\right] \, d\tau.
\end{align*}
As before, putting
$M(s) := \sup_{s_0 \le \tau \le s} E_5(\tau)$
and noting that $\Lambda(s)$ is bounded by both above and below,
we see that
\begin{align*}
	M(s)
		+ \int_{s_0}^s L_4(\tau)\,d\tau
	&\le C_2 \varepsilon^2  \| (v_0, w_0) \|_{H^{1,m}\times H^{0,m}}^2
		+ C_2 \sum_{\substack{i=1,\ldots, k\\ p_{i3}=1}}
				M(s)^{p_{i1}+p_{i2}} \int_{s_0}^s
				L_4(\tau)^{p_{i3}}d\tau \\
	&\quad + C_2 \sum_{i=1}^k
		\left( M(s)^{p_{i1}+p_{i2}+p_{i3}}
			+ M(s)^{(p_{i1}+p_{i2}+p_{i3}+1)/2} \right)
\end{align*}
for $s \ge s_0$ with some constant $C_2 >0$. 
Taking $\varepsilon_1$ sufficiently small so that
\begin{align*}
	2C_2 \varepsilon^2 I_0
		+ \int_{s_0}^s L_4(\tau)\,d\tau
	&> C_2 \varepsilon^2 I_0
		+ C_2 \sum_{\substack{i=1,\ldots, k\\ p_{i3}=1}}
			(2C_2 \varepsilon^2 I_0)^{p_{i1}+p_{i2}} \int_{s_0}^s
				L_4(\tau)^{p_{i3}}d\tau \\
	&\quad + C_2 \sum_{i=1}^k
		\left( (2C_2 \varepsilon^2 I_0)^{p_{i1}+p_{i2}+p_{i3}}
		+ (2C_2 \varepsilon^2 I_0)^{(p_{i1}+p_{i2}+p_{i3}+1)/2} \right)
\end{align*}
holds for $\varepsilon \in (0,\varepsilon_1]$,
where $I_0 = \| (v_0, w_0) \|_{H^{1,m}\times H^{0,m}}^2$.
Combining this with
$M(s_0) \le C_2 \varepsilon^2 I_0$
and the continuity of
$M(s)$ with respect to $s$, we conclude that
\begin{align*}
	M(s) \le 2 C_2 \varepsilon^2 \| (v_0, w_0) \|_{H^{1,m}\times H^{0,m}}^2,
\end{align*}
which leads to \eqref{aprivw}
and completes the proof of Proposition \ref{prop_ap} for $n=1$.

\subsection{Proof of Theorem \ref{thm}: asymptotic behavior}
Next, we prove the asymptotic behavior \eqref{sol_asym}.
For simplicity, we only consider the case $n \ge 2$,
since the proof of the one-dimensional case is similar.
Putting
\begin{align}
\label{lambda}
	\lambda = \min \left\{ \frac{1}{2}, \frac{m}{2}-\frac{n}{4},
		\lambda_0,\lambda_1 \right\} -\eta,
\end{align}
where $\eta>0$ is an arbitrary small number,
and $\lambda_0,\lambda_1$ are defined by
\eqref{lambda_0}, \eqref{lambda_1},
and turning back to \eqref{en3_est}
and using Lemma \ref{lem_rema} with $\tilde{\eta}=\frac{1}{2}$ to $R_4(s)$,
we have
\begin{align*}
	\frac{d}{ds}E_4(s) + 2 \lambda E_4(s)
		+ \frac{1}{2}L_4(s)
		&\le C e^{-2\lambda_0 s} E_5(s)
		+ C e^{-2\lambda_1 s} E_5(s)^p\\
		&\le C e^{-2\lambda_2 s} \varepsilon^2 \| (u_0, u_1) \|_{H^{1,m}\times H^{0,m}}^2,
\end{align*}
where
$\lambda_2 = \min\{ \lambda_0, \lambda_1 \}$.
Multiplying the above inequality by $e^{2 \lambda s}$, we obtain
\begin{align*}
	\frac{d}{ds}\left[ e^{2 \lambda s} E_4(s) \right]
		+ \frac{e^{2 \lambda s}}{2} L_4(s)
		\le C e^{-2 \eta s} \varepsilon^2 \| (u_0, u_1) \|_{H^{1,m}\times H^{0,m}}^2.
\end{align*}
Integrating it over $[s_0, s]$ and using Lemma \ref{lem_l4}, we have
\begin{align*}
	E_4(s)
		+ \int_{s_0}^s e^{-2\lambda (s-\tau)}
			\left( \| f(s) \|_{H^{1,m}}^2 + \| g(s) \|_{H^{0,m}}^2
				+ \left( \frac{d\alpha}{d\tau}(\tau) \right)^2 \right) d\tau
		\le Ce^{-2\lambda s}
			\varepsilon^2 \| (u_0, u_1) \|_{H^{1,m}\times H^{0,m}}^2.
\end{align*}
In particular, for $s_0 \le \tilde{s} \le s$, one has 
\begin{align*}
	| \alpha(s)- \alpha(\tilde{s}) |^2
	&= \left( \int_{\tilde{s}}^s \left( \frac{d\alpha}{d\tau}(\tau) \right)^2 d\tau \right)^2\\
	&\le \left( \int_{\tilde{s}}^s e^{- 2\lambda \tau} d\tau \right)
		\left( \int_{\tilde{s}}^s e^{ 2\lambda \tau}
			\left( \frac{d\alpha}{d\tau}(\tau) \right)^2 d\tau \right)\\
	& \le Ce^{-2\lambda \tilde{s}}
		\varepsilon^2 \| (u_0, u_1) \|_{H^{1,m}\times H^{0,m}}^2,
\end{align*}
and hence, the limit
$\alpha^{\ast} = \lim_{s\to +\infty}\alpha(s)$
exists and it follows that
\begin{align*}
	| \alpha(s) - \alpha^{\ast} |^2
		\le C e^{-2\lambda s}
		\varepsilon^2 \| (u_0, u_1) \|_{H^{1,m}\times H^{0,m}}^2.
\end{align*}
Finally, we have
\begin{align*}
	\| v(s) - \alpha^{\ast}\varphi_0 \|_{H^{1,m}}^2
	\le \| f(s) \|_{H^{1,m}}^2 + |\alpha(s) - \alpha^{\ast} |^2 \| \varphi_0 \|_{H^{1,m}}^2
	\le C e^{-2\lambda s} \varepsilon^2 \| (u_0, u_1) \|_{H^{1,m}\times H^{0,m}}^2.
\end{align*}
Recalling the relation \eqref{uvw} and
$(B(t)+1)^{-n/2} \varphi_0((B(t)+1)^{-1/2}x) = \mathcal{G}(B(t)+1,x)$,
where $\mathcal{G}$ is the Gaussian defined by \eqref{gauss},
we obtain
\[
	\| u(t,\cdot ) - \alpha^{\ast} \mathcal{G}(B(t)+1, \cdot) \|_{L^2}^2
	\le C \varepsilon^2 (B(t)+1)^{-n/2-2\lambda} \| (u_0, u_1) \|_{H^{1,m}\times H^{0,m}}^2,
\]
which completes the proof of Theorem \ref{thm}.


\section{Estimates of the remainder terms}
In this section, we give a proof to Lemma \ref{lem_rema}.

\begin{lemma}\label{lem_nl}
Under the assumptions \eqref{b0}, \eqref{N1} and \eqref{N2}, we have
\begin{align}
\nonumber
	&\left\| e^{3s/2}
		 	N\left( e^{-s/2}v,
						e^{-s} v_y,
						b(t(s))^{-1}e^{-3s/2}w \right) \right\|_{H^{0,1}}^2\\
\label{est_nl1}
	&\quad\le C e^{- 2 \lambda_1 s}
		\sum_{i=1}^k
			\left( \|f(s)\|_{H^{1,1}} + \alpha(s) \right)^{2(p_{i1}+p_{i2})}
			\left( \|g(s)\|_{H^{0,1}} + \alpha(s) + \frac{d\alpha}{ds}(s) \right)^{2p_{i3}}
\end{align}
for $n=1$, $s \ge 0$ and
\begin{align}
\label{est_nl2}
	\left\| e^{\left(\frac{n}{2}+1\right)s}
		 	N\left( e^{-\frac{n}{2}s}v \right) \right\|_{H^{0,m}}^2
		\le C e^{- 2 \lambda_1 s}
		\left( \|f(s)\|_{H^{1,m}} + \alpha(s)  \right)^{2p}
\end{align}
for $n\ge 2$, $s \ge 0$, where $\lambda_1$ is given by \eqref{lambda_1}.
\end{lemma}
\begin{proof}
When $n=1, \beta\in (-1,1)$, by the assumption \eqref{N1} and Lemma \ref{lem_b0b1},
we compute
\begin{align*}
	(1+y^2)e^{3s}N_i \left( e^{-s/2}v, e^{-s}v_y, b(t(s))^{-1}e^{-3s/2}w \right)^2 
	\le C (1+y^2) e^{- 2 \lambda_1 s} | v |^{2p_{i1}} |v_y |^{2p_{i2}} |w|^{2p_{i3}},
\end{align*}
where $\lambda_1$ is defined by \eqref{lambda_1}.
By the Sobolev inequality $\| v(s) \|_{L^{\infty}} \le C \| v(s) \|_{H^{1,0}}$,
we calculate
\begin{align*}
	&(1+y^2) e^{- 2 \lambda_1 s} | v |^{2p_{i1}} |v_y |^{2p_{i2}} |w|^{2p_{i3}} \\
	&\quad \le C e^{- 2 \lambda_1 s} | v^2 |^{p_{i1}+p_{i2}+p_{i3}-1}
		( (1+y^2)v^2 )^{1-p_{i2}-p_{i3}} ((1+y^2)v_y^2)^{p_{i2}} ((1+y^2)w^2)^{p_{i3}}\\
	&\quad \le C e^{- 2 \lambda_1 s} \| v(s) \|_{H^{1,0}}^{2( p_{i1}+p_{i2}+p_{i3}-1 )}
		( (1+y^2)v^2 )^{1-p_{i2}-p_{i3}} ((1+y^2)v_y^2)^{p_{i2}} ((1+y^2)w^2)^{p_{i3}}.
\end{align*}
Therefore, by the H\"{o}lder inequality, we conclude
\begin{align*}
	&\left\| e^{3s/2} N_i
		\left( e^{-s/2}v, e^{-s}v_y, b(t(s))^{-1}e^{-3s/2}w \right) \right\|_{H^{0,1}}^2 \\
	&\quad \le C  e^{- 2 \lambda_1 s}
	\| v(s) \|_{H^{1,0}}^{2( p_{i1}+p_{i2}-1 )}
	\| v(s) \|_{H^{1,1}}^{2(1-p_{i2})}
		\| v(s) \|_{H^{1,1}}^{2p_{i2}} \| w(s) \|_{H^{0,1}}^{2p_{i3}}\\
	&\quad \le C e^{- 2 \lambda_1 s}
			\left( \|f(s)\|_{H^{1,1}} + \alpha(s) \right)^{2(p_{i1}+p_{i2})}
			\left( \|g(s)\|_{H^{0,1}} + \alpha(s) + \frac{d\alpha}{ds}(s) \right)^{2p_{i3}}
\end{align*}

When $n=1, p_{i3}\neq 0, \beta = -1$, we obtain
\begin{align*}
	&(1+y^2)e^{3s}N_i \left( e^{-s/2}v, e^{-s}v_y, b(t(s))^{-1}e^{-3s/2}w \right)^2 \\
	&\quad \le C (1+y^2) e^{(3-p_{i1}-2p_{i2}-3p_{i3})s}
		b(t(s))^{-p_{i3}} | v |^{2p_{i1}} |v_y |^{2p_{i2}} |w|^{2p_{i3}}\\
	&\quad \le C (1+y^2) e^{-\lambda_{i1}s} | v |^{2p_{i1}} |v_y |^{2p_{i2}} |w|^{2p_{i3}},
\end{align*}
where we can take $\lambda_{i1}$ as an arbitrary large number,
since Lemma \ref{lem_b0b1} shows $b(t(s))^{-p_{i3}} \sim \exp ( -p_{i3} e^s )$.
Therefore, by the same way, we obtain the desired estimate.

Next, we consider the case $n\ge 2$.
By the assumption \eqref{N2} and
Lemma \ref{lem_gn}, we have
\begin{align*}
	\left\| e^{\left(\frac{n}{2}+1\right)s}
		 	N\left( e^{-\frac{n}{2}s}v \right) \right\|_{H^{0,m}}^2
	&\le
	C \int_{\mathbb{R}^n}
	e^{2\left(\frac{n}{2}+1\right)s}
	\langle y \rangle^{2m} \left| e^{-\frac{n}{2}s}v(s,y) \right|^{2p} dy \\
	&\le
	C e^{- 2 \lambda_1s} \int_{\mathbb{R}^n}
		\left| \langle y \rangle^{m/p} v(s,y) \right|^{2p} dy\\
	&\le 
	C e^{- 2 \lambda_1s}
	\left\| \nabla \left( \langle y \rangle^{m/p} v(s) \right) \right\|_{L^2}^{2p\sigma}
	\left\| \langle y \rangle^{m/p} v(s) \right\|_{L^2}^{2p(1-\sigma)} \\
	&\le
	C e^{- 2 \lambda_1s}
	\| v(s) \|_{H^{1,m}}^{2p}\\
	&\le
	C e^{-2 \lambda_1s} \left( \| f(s) \|_{H^{1,m}} + \alpha(s) \right)^{2p},
\end{align*}
which completes the proof.
\end{proof}

From Lemmas \ref{lem_b0b1}, \ref{lem_nl} and the assumption \eqref{c},
we immediately obtain the following estimate:

\begin{lemma}\label{lem_r}
Let $r$ be defined by \eqref{r}.
Under the assumptions \eqref{b0}--\eqref{N2}, we have
\begin{align*}
	\|  r(s) \|_{H^{0,m}}^2
	&\le
	C e^{-2\lambda_0 s}
		\left( \| f(s) \|_{H^{1,m}}^2 + \| g(s) \|_{H^{0,m}}^2 + \alpha(s)^2 + \left( \frac{d\alpha}{ds}(s) \right)^2
		\right)\\
	&\quad + C e^{- 2 \lambda_1 s}
		\sum_{i=1}^k
			\left( \|f(s)\|_{H^{1,1}} + \alpha(s) \right)^{2(p_{i1}+p_{i2})}
			\left( \|g(s)\|_{H^{0,1}} + \alpha(s) + \frac{d\alpha}{ds}(s) \right)^{2p_{i3}}
\end{align*}
for $n=1$, $s\ge 0$ and
\begin{align*}
	\|  r(s) \|_{H^{0,m}}^2
	&\le
	C e^{-2\lambda_0 s}
		\left( \| f(s) \|_{H^{1,m}}^2 + \| g(s) \|_{H^{0,m}}^2 + \alpha(s)^2
			+ \left( \frac{d\alpha}{ds}(s) \right)^2
		\right)\\
	&\quad +C e^{- 2 \lambda_1 s}
		\left( \|f(s)\|_{H^{1,m}} + \alpha(s)  \right)^{2p}
\end{align*}
for $n\ge 2$, $s\ge 0$, where
$\lambda_0, \lambda_1$ are defined by \eqref{lambda_0}, \eqref{lambda_1}, respectively.
\end{lemma}
\begin{proof}
By Lemma \ref{lem_nl}, it suffices to estimate
\[
	\frac{1}{b(t(s))^2}\frac{db}{dt}(t(s)) w +e^{s/2}c(t(s))\cdot\nabla_yv+e^{s}d(t(s))v.
\]
Applying Lemma \ref{lem_b0b1}, we have
\begin{align*}
	\left\| \frac{1}{b(t(s))^2} \frac{db}{dt}(t(s)) w(s) \right\|_{H^{0,m}}^2
	&\le
	C \left( \| g(s) \|_{H^{0,m}}^2 + \alpha(s)^2 + \left( \frac{d\alpha}{ds}(s) \right)^2 \right)
	\times 
	\begin{cases}
	e^{-2(1-\beta) s/(1+\beta)}& \beta \in (-1,1),\\
	\exp(-4e^{s}) &\beta = -1.
	\end{cases}
\end{align*}
Also, the assumption \eqref{c} implies
\begin{align*}
	\left\| e^{s/2}c(t(s))\cdot \nabla_y v(s) \right\|_{H^{0,m}}^2
	&\le
		C \left( \| f(s) \|_{H^{1,1}}^2 + \alpha(s)^2 \right) \times
	\begin{cases}
	e^{- \left( (2\gamma)/(1+\beta) - 1 \right)s}&\beta\in (-1,1),\\
	\exp\left( -2\gamma e^{s} + s \right)& \beta = -1
	\end{cases}
\end{align*}
and
\begin{align*}
	\left\| e^s d(t(s))v(s) \right\|_{H^{0,m}}^2
	&\le
	C \left( \| f(s) \|_{H^{1,1}}^2 + \alpha(s)^2 \right) \times
	\begin{cases}
	e^{-\left( 2\nu/(1+\beta) - 2 \right)}&\beta \in (-1,1),\\
	\exp \left( -2\nu e^{s} + 2s \right) &\beta =-1.
	\end{cases}
\end{align*}
Summing up the above estimates and \eqref{est_nl1}, \eqref{est_nl2},
we obtain the desired estimate.
\end{proof}

Next, we estimate the term $h$ given by \eqref{h}.
By Lemmas \ref{lem_b0b1} and \ref{lem_r}, we can easily have the
following estimate:
\begin{lemma}\label{lem_h}
Let $h$ be defined by \eqref{h}.
Under the assumption \eqref{b0}--\eqref{N2}, we have
\begin{align*}
	\| h(s) \|_{H^{0,m}}^2
	&\le
	C e^{-2\lambda_0 s}
		\left( \| f(s) \|_{H^{1,m}}^2 + \| g(s) \|_{H^{0,m}}^2 + \alpha(s)^2 + \left( \frac{d\alpha}{ds}(s) \right)^2
		\right)\\
	&\quad + C e^{- 2 \lambda_1 s}
		\sum_{i=1}^k
			\left( \|f(s)\|_{H^{1,1}} + \alpha(s) \right)^{2(p_{i1}+p_{i2})}
			\left( \|g(s)\|_{H^{0,1}} + \alpha(s) + \frac{d\alpha}{ds}(s) \right)^{2p_{i3}}
\end{align*}
for $n=1$, $s\ge 0$ and
\begin{align*}
	\|  h(s) \|_{H^{0,m}}^2
	&\le
	C e^{-2\lambda_0 s}
		\left( \| f(s) \|_{H^{1,m}}^2 + \| g(s) \|_{H^{0,m}}^2
				+ \alpha(s)^2 + \left( \frac{d\alpha}{ds}(s) \right)^2
		\right)\\
	&\quad +C e^{- 2 \lambda_1 s}
		\left( \|f(s)\|_{H^{1,m}} + \alpha(s)  \right)^{2p}
\end{align*}
for $n\ge 2$, $s\ge 0$,
where
$\lambda_0, \lambda_1$ are defined by \eqref{lambda_0}, \eqref{lambda_1}, respectively.
\end{lemma}
\begin{proof}
We easily estimate
\begin{align*}
	&\left\| \frac{e^{-s}}{b(t(s))^2}
		\left( -2 \frac{d\alpha}{ds}(s) \psi_0(y)
		+\alpha(s)\left(\frac{y}{2}\cdot\nabla_y\psi_0(y)
			+\left(\frac{n}{2}+1\right)\psi_0(y) \right) \right)
	\right\|_{H^{0,m}}^2 \\
	&\quad \le
	C e^{-2\lambda_0 s}
		\left( \alpha(s)^2 + \left( \frac{d\alpha}{ds}(s) \right)^2
		\right).
\end{align*}
For the term
$r(s)$, we apply Lemma \ref{lem_r}.
Finally, for the term
$(\int_{\mathbb{R}^n} r(s,y)\,dy) \varphi_0(y)$,
we note that
\begin{align}
\label{intr}
	 \left| \int_{\mathbb{R}^n} r(s,y) dy \right|
	\le C \| r(s) \|_{H^{0,m}},
\end{align}
holds due to $m> n/2$.
Thus, we apply Lemma \ref{lem_r} again to obtain the conclusion.
\end{proof}

Moreover, combining \eqref{h_int} and the Hardy-type inequalities \eqref{hardy}, \eqref{hardy2},
we also have
\begin{lemma}\label{lem_caph}
Let $H$ be defined by \eqref{H}.
Under the assumption \eqref{b0}--\eqref{N2}, we have
\begin{align*}
	\| H(s) \|_{H^{0,m}}^2
	&\le
	C e^{-2\lambda_0 s}
		\left( \| f(s) \|_{H^{1,m}}^2 + \| g(s) \|_{H^{0,m}}^2
			+ \alpha(s)^2 + \left( \frac{d\alpha}{ds}(s) \right)^2
		\right)\\
	&\quad + C e^{- 2 \lambda_1 s}
		\sum_{i=1}^k
			\left( \|f(s)\|_{H^{1,1}} + \alpha(s) \right)^{2(p_{i1}+p_{i2})}
			\left( \|g(s)\|_{H^{0,1}} + \alpha(s) + \frac{d\alpha}{ds}(s) \right)^{2p_{i3}}
\end{align*}
for $n=1$, $s\ge 0$ and
\begin{align*}
	\|  H(s) \|_{H^{0,m}}^2
	&\le
	C e^{-2\lambda_0 s}
		\left( \| f(s) \|_{H^{1,m}}^2 + \| g(s) \|_{H^{0,m}}^2
				+ \alpha(s)^2 + \left( \frac{d\alpha}{ds}(s) \right)^2
		\right)\\
	&\quad +C e^{- 2 \lambda_1 s}
		\left( \|f(s)\|_{H^{1,m}} + \alpha(s)  \right)^{2p}
\end{align*}
for $n\ge 2$, $s\ge 0$,
where
$\lambda_0, \lambda_1$ are defined by \eqref{lambda_0}, \eqref{lambda_1}, respectively.\end{lemma}

Now we are at the position to prove Lemma \ref{lem_rema}.
\begin{proof}[Proof of Lemma \ref{lem_rema}]
We first prove the estimate for $R_4(s)$.
Let $\tilde{\eta}>0$ be an arbitrary small number.
Then, by the Schwarz inequality
and Lemmas \ref{lem_l4} and \ref{lem_compatible},
there exists $s_3 \ge s_2$ such that
the terms not including the nonlinearity are easily bounded by
$\tilde{\eta} L_4(s) + C(\tilde{\eta}) e^{-2\lambda_0 s} E_5(s)$
for
$s \ge s_3$.
The terms including the nonlinearity consist of the following
three terms:
\begin{align*}
	\int_{\mathbb{R}^n} (F+G)H\, dy,\quad
	\int_{\mathbb{R}^n} (1+|y|^{2m}) (f+g)h\,dy,\quad
	\left( \int_{\mathbb{R}^n} r(s,y)\,dy \right) \frac{d\alpha}{ds}(s).
\end{align*}
By the Schwarz inequality and Lemmas
\ref{lem_hardy}, \ref{lem_hardy2}, \ref{lem_l4}, \ref{lem_compatible}, \ref{lem_h}, \ref{lem_caph},
there exists $s_4 \ge s_2$ such that
the first two terms are easily bounded by
\begin{align*}
	\tilde{\eta} L_4(s) + C(\tilde{\eta}) e^{-2\lambda_0 s} E_5(s)
	+ \begin{cases}
		\displaystyle C e^{- 2 \lambda_1 s}
		\sum_{i=1}^k
			E_5(s)^{p_{i1}+p_{i2}}
			\left( E_5(s)^{p_{i3}} + L_4(s)^{p_{i3}} \right)&(n=1),\\
		C e^{- 2 \lambda_1 s}
		\left( \|f(s)\|_{H^{1,m}} + \alpha(s)  \right)^{2p}&(n \ge 2)
	\end{cases}
\end{align*}
for $s\ge s_4$.
For the third term, we apply the Schwarz inequality to obtain
\[
	\left| \left( \int_{\mathbb{R}^n} r(s,y)\,dy \right) \frac{d\alpha}{ds}(s) \right|
	\le \tilde{\eta} \left( \frac{d\alpha}{ds}(s) \right)^2
		+ C(\tilde{\eta}) \left( \int_{\mathbb{R}^n} r(s,y)\,dy \right)^2.
\]
Noting \eqref{intr} and
applying Lemma \ref{lem_r}, and then Lemmas \ref{lem_l4} and \ref{lem_compatible},
we have the desired estimate.

Finally, we prove the estimate for $R_5(s)$.
Let $\tilde{\eta}>0$ be an arbitrary small number.
Recall that  $R_5(s) = R_4(s) + \tilde{R}_5(s)$ with
\[
	\tilde{R}_5(s) = \frac{e^{-s}}{b(t(s))^2}\left( \frac{d\alpha}{ds}(s) \right)^2
		- \frac{2}{b(t(s))^2}\alpha(s) \frac{db}{dt}(t(s)) \frac{d\alpha}{ds}(s)
		+\alpha(s) \left( \int_{\mathbb{R}^n} r(s,y)\,dy \right).
\]
We have already estimated $R_4(s)$ and hence, it suffices to estimate $\tilde{R}_5(s)$.
By Lemmas \ref{lem_l4} and \ref{lem_compatible},
there exists $s_5 \ge s_2$ such that
the first two terms are easily estimated by
$\tilde{\eta} L_4(s) + C(\tilde{\eta}) e^{-2\lambda_0 s} E_5(s)$
for
$s \ge s_5$.
Moreover, by \eqref{intr},
$\alpha(s) \le C E_5(s)^{1/2}$ and Lemma \ref{lem_r}, 
there exists $s_6 \ge s_2$ such that the third term is estimated as
\begin{align*}
	\alpha(s) \left| \int_{\mathbb{R}^n} r(s,y)dy \right|
	&\le \alpha(s) \| r(s) \|_{H^{0,m}} \\
	&\le \tilde{\eta} L_4(s) + C e^{-\lambda_0 s} E_5(s) \\
	&\quad + \begin{cases}
		\displaystyle C e^{-\lambda_1 s}
			\sum_{i=1}^k
			E_5(s)^{(p_{i1}+p_{i2}+1)/2}
			\left( E_5(s)^{p_{i3}/2} +  L_4(s)^{p_{i3}/2} \right) &(n=1),\\
		C e^{-\lambda_1 s}
			E_5(s)^{(p+1)/2} &(n\ge 2)
	\end{cases}
\end{align*}
for
$s \ge s_6$.
When $n=1$,
we further apply the Schwarz inequality to
the terms in the sum corresponding to $p_{i3} =1$
and obtain
\[
	e^{-\lambda_1 s}E_5(s)^{(p_{i1}+p_{i2}+1)/2} L_4(s)^{p_{i3}/2}
	\le \tilde{\eta} L_4(s)
		+ C(\tilde{\eta}) e^{-2\lambda_1s} E_5(s)^{(p_{i1}+p_{i2}+p_{i3})}.
\]
Finally, letting
$s_0 := \max\{ s_3, s_4, s_5, s_6 \}$
and combining the above estimates,
we have the conclusion.
\end{proof}

\section*{Acknowledgements}
The author is deeply grateful to Professor Yoshiyuki Kagei
for useful suggestions and constructive comments.
He appreciates helpful comments and generous support by
Professor \mbox{Mitsuru} \mbox{Sugimoto}.
He would also like to express his
gratitude to an anonymous referee
for careful reading of the manuscript and
many useful suggestions and comments. 
This work is supported by
Grant-in-Aid for JSPS Fellows 15J01600
of Japan Society for the Promotion of Science.


\end{document}